\documentclass[12pt]{article}
\usepackage{amsmath,amsfonts,amssymb,amsthm}
\usepackage[british]{babel}
\usepackage{color}
\usepackage{hyperref}
\usepackage{dsfont}

\newcommand{\fn}{\mathfrak{n}}
\newcommand{\B}{\mathcal{B}}
\newcommand{\F}{\mathcal{F}}
\newcommand{\calF}{\mathcal{F}}
\newcommand{\calE}{\mathcal{E}}

\newcommand{\calS}{\mathcal{S}}

\newcommand{\R}{\mathbb{R}}
\newcommand{\calR}{\mathcal{R}}

\newcommand{\bbS}{\mathbb{S}}

\newcommand{\DIV}{\textnormal{div}\,}
\newcommand{\eps}{\varepsilon}
\newcommand{\bbR}{\mathbb{R}}
\newcommand{\bbI}{\mathbb{I}}
\newcommand{\lra}{\longrightarrow}

\def\longrightharpoonup{\relbar\joinrel\rightharpoonup}
\def\cv{\stackrel{w}{\longrightharpoonup}}
\def\cvwstar{\stackrel{w*}{\longrightharpoonup}}

\allowdisplaybreaks

\newtheorem{Theorem}{Theorem}
\newtheorem{Definition}{Definition}

\newtheorem{Proposition}{Proposition}
\newtheorem{Lemma}{Lemma}
\newtheorem{Remark}{Remark}
\begin{document}
	
\date{\today}
\title{Ad hoc test functions for homogenization of compressible viscous fluid with application to the obstacle problem in dimension two}
\author{Marco Bravin \footnote{Delft Institute of Applied Mathematics, Delft University of Technology, Mekelweg 4
2628 CD Delft, The Netherlands } \\ \url{M.Bravin@tudelft.nl}}

\maketitle
	
\begin{abstract}
In this paper we highlight a set of ad hoc test functions to study the homogenization of viscous compressible fluid in domains with tiny holes. This set of functions allows to improve previous results in dimensions two and three. As an application we show that the presence of a small obstacle does not influence the dynamics of a viscous compressible fluid in dimension two.  
\end{abstract}

\section{Introduction}

The study of the interaction of some tiny holes and a compressible viscous fluid has been widely studied, see for instance \cite{fei:lu}-\cite{lu:sch}-\cite{P:S}-\cite{N:P}-\cite{ION}-\cite{BO2}, where different regime where considered. In the case of homogenization with number of holes comparable to $ \eps^{-d} $ where $ d $ is the dimension, these types of results require in general a lower-bound on $ \gamma $  and $ \alpha $ which appear respectively in the pressure law $p(\rho) = \rho^{\gamma}$ and on the size $ \eps^{\alpha} $ of the holes. These limitations on $ \gamma $ and $ \alpha $ are used in two points. The first one is to show improved pressure estimates independent of $ \eps $. The second one is to verify that the limit of the solutions of the compressible Navier-Stokes equations in the domain with tiny holes satisfies the same system but in the domain without holes.  For example in \cite{lu:sch} the limitations are
\begin{equation*}   
\gamma > 6, \quad \alpha >  3  \quad \text{ and } \quad  \frac{\gamma - 6}{2\gamma - 3} \alpha > 3, 
\end{equation*}
in \cite{fei:lu} they are $ \gamma > 2 $,  $ \alpha > 3 $ and $ \alpha(\gamma-2) > 2\gamma -3 $.  Let us notice that the more severe hypothesis are used to verify that the limit solve the compressible Navier-Stokes equations, in fact in  \cite{fei:lu}-\cite{lu:sch}-\cite{P:S}-\cite{N:P}-\cite{ION}, the authors multiply the test functions for the limiting system by cut-off to make them compatible with domains with tiny holes. Let now for example consider the case of the non-stationary compressible Navier-Stokes system outside a tiny hole $ \calS_{\eps} = \eps \calS $ in dimension three. Following the strategy used in \cite{lu:sch}, to show that the limit satisfies the compressible Navier-Stokes is enough to multiply any smooth test function $ \varphi  $ by $  \eta_{\eps} $ a scaling cut-off of the type $ \eta_{\eps} = \eta(x/\eps) $ where $ \eta_{1} $ is $ 1$ in $ \calS $ and $ 0 $ outside $ 2 \calS $. In this way  $ \varphi \eta_{\eps} $ is an admissible test function for the domain with the hole $ \calS_{\eps} $ and it remains to pass to the limit in the weak formulation. The term that gives the limitation  is 
\begin{equation}
\label{pressure}       
\int \rho^{\gamma}_{\eps} \DIV( \eta_{\eps} \varphi ) = \int \rho^{\gamma}_{\eps} \eta_{\eps} \DIV(\varphi) + \int \rho^{\gamma}_{\eps} \varphi \cdot \nabla \eta_{\eps}.  
 \end{equation}
The difficult term to tackle is the second one on the right hand side. Notice that $ \|\nabla \eta_{\eps}\|_{L^p}  $ converges to $ 0 $  for $ p < 3 $. We can then conclude only if we are able to show a uniform bound for $ \rho_{\eps} $ in $ L^{q} $ with $ q  > 3\gamma/2 $. This condition together with the fact that the improved pressure estimate holds for $ \gamma + \theta \leq 5 \gamma / 3 - 1 $, we deduce the limitation 
$$ \frac{5}{3} \gamma - 1 >  \frac{3}{2}\gamma \quad \text{ if and only if  } \quad \gamma >   6. $$ 

The situation is even worst in the case of dimension two because  $ \|\nabla \eta_{\eps}\|_{L^p} \lra 0 $ only for $ p < 2 $ and the improve pressure estimate holds for $ \gamma + \theta  <  2 \gamma  - 1 $. In particular $ 2\gamma -1 > 2 \gamma  $ is false for any $ \gamma $. For this reason there are not results on the homogenisation of unsteady compressible Navier-Stokes equations in this setting when the dimension is two.  

To avoid this issue in dimension two, we introduce the ad hoc test function  
$$ \Phi_{\eps}[\varphi] = \eta_{\eps} \varphi + \nabla^{\perp} \eta_{\eps} x^{\perp} \cdot \varphi. $$
This function has much better property because if we define 
$$ \tilde{\Phi}_{\eps}^0[\varphi] = (1-\eta_{\eps}) \varphi(0) -  \nabla^{\perp} \eta_{\eps} x^{\perp} \cdot \varphi(0), $$
we have 
$$ \DIV(\tilde{\Phi}_{\eps}^0[\varphi]  ) = 0. $$ 
This allowed to rewrite 
\begin{align*}
\DIV(\Phi_{\eps}[\varphi]) - \eta_{\eps} \DIV(\varphi) =  & \, \DIV(\Phi_{\eps}[\varphi]) + \DIV(\tilde{\Phi}_{\eps}^0[\varphi]  )  - \eta_{\eps} \DIV(\varphi) \\ = & \, \nabla \eta_{\eps} (\varphi - \varphi(0))  + \nabla^{\perp}\eta_{\eps} \otimes (\varphi - \varphi(0)) : \nabla x^{\perp} + \nabla^{\perp} \eta_{\eps} \otimes x^{\perp} : \nabla \varphi.
\end{align*}
If for example $ \varphi $ is Lipschitz the terms of the type
\begin{equation*}
\nabla \eta_{\eps} (\varphi - \varphi(0))  = x \nabla \eta_{\eps} \frac{(\varphi(x) - \varphi(0))}{x}
\end{equation*}
converges to zero in any $ L^p $ with $ p < + \infty $. In particular we have that  
\begin{align*}       
\int \rho^{\gamma}_{\eps} \DIV( \Phi_{\eps}[\varphi]) = & \, \int \rho^{\gamma}_{\eps }\eta_{\eps} \DIV(\varphi) + \int \rho^{\gamma}_{\eps}  (\DIV(\Phi_{\eps}[\varphi]) - \eta_{\eps} \DIV(\varphi)) \\
& \, \lra    \int \overline{\rho^{\gamma}} \DIV(\varphi),
 \end{align*}
if we have a uniform bound of $ \rho_{\eps}  $ in $ L^ p $ for some $ p > \gamma$.

In dimension three a possible set of ad hoc test functions is
\begin{equation*}
\Phi_{\eps}[\varphi] = \eta_{\eps} \varphi +\begin{pmatrix} x_2 \partial_2 \eta_{\eps} \varphi_1 - x_2 \partial_3 \eta_{\eps}\varphi_3 \\ x_3 \partial_3 \eta_{\eps} \varphi_2 - x_2 \partial_1 \eta_{\eps} \varphi_1 \\ x_1 \partial_1 \eta_{\eps} \varphi_{3} - x_{3} \partial_2 \eta_{\eps} \varphi_2 \end{pmatrix},
\end{equation*}
with 
\begin{equation*}
\tilde{\Phi}^0_{\eps}[\varphi] = (1-\eta_{\eps}) \varphi(0) - \begin{pmatrix} x_2 \partial_2 \eta_{\eps} \varphi_1(0) - x_2 \partial_3 \eta_{\eps}\varphi_3(0) \\ x_3 \partial_3 \eta_{\eps} \varphi_2(0) - x_2 \partial_1 \eta_{\eps} \varphi_1(0) \\ x_1 \partial_1 \eta_{\eps} \varphi_{3}(0) - x_{3} \partial_2 \eta_{\eps} \varphi_2(0) \end{pmatrix} .
\end{equation*}

With the help of ad hoc test functions $ \Phi_{\eps}[\varphi] $, we can improved the hypothesis on $ \gamma $ and $ \alpha $ in the study of homogenization of compressible viscous fluid with tiny holes.  Moreover the restrictions on the parameters will not come from the pressure term, in particular we expect that \cite{lu:sch} can be shown in the case $ \gamma > 3 $ and $ \alpha > \max\{ 3, (2\gamma-3)/(\gamma-3) \}$ and these results can be extended also in the case of dimension two for appropriate lower-bounds of $ \gamma $ and $ \alpha$.

To verify that the idea introduced in this section works, we apply it to a simpler problem that is the obstacle problem in dimension two. In particular we show that the presence of a small hole does not influence the dynamic of a viscous compressible fluid in dimension two.

 Let us recall that in the case the hole is replaced by a rigid body, in \cite{FRZ}, the authors show that the small object does not influence the dynamics of a viscous compressible fluid in dimension three under the hypothesis $ \gamma > 3/ 2 $. Finally let us mention that homogenisation of compressible Navier-Stokes has been studied also in the setting of randomly perforated domains in \cite{BO}.

\section{The obstacle problem in dimension two}

At a mathematical level, we consider $ \Omega \subset \bbR^2 $ a open, bounded, connected, simply-connected subset of $ \bbR^2 $ with Lipschitz boundary such that $ 0 \in \Omega $. For a small parameter $ \eps > 0 $, we consider a sequence of small holes $ \calS_{\eps} \in B_{\eps}(0) \subset \Omega $ such that they are open, connected, simply-connected and with Lipschitz boundary. The fluid domain is $ \calF_{\eps} = \Omega \setminus \calS_{\eps} $ and to model a viscous compressible fluid in $ \calF_{\eps} $, we consider the compressible Navier-Stokes equations that reads
\begin{align}
\partial \rho_{\eps} + \DIV(\rho_{\eps} u_{\eps}) = & \, 0 \quad && \text{ for } x \in \calF_{\eps},  \nonumber \\
\partial_t (\rho_{\eps} u_{\eps}) + \DIV(\rho_{\eps} u_{\eps} \otimes u_{\eps}) - \DIV(\bbS(u_{\eps}))  -\nabla  p(\rho_{\eps}) =  & \, 0 \quad && \text{ for } x \in \calF_{\eps},  \label{CNS:equ} \\
u_{\eps} = \, &  0 \quad && \text{ for } x \in \partial \calF_{\eps}, \nonumber \\ 
\rho_{\eps}(0,.) = \rho^{in}_{\eps}, \quad (\rho_{\eps} u_{\eps})(0) = & \, q^{in} \quad && \text{ for } x \in \calF_{\eps}, \nonumber
\end{align}  
where $ u_{\eps} : \bbR^{+} \times \calF_{\eps} \lra \bbR^2 $ describes the velocity of the fluid, $ \rho_{\eps} : \bbR^{+} \times \calF_{\eps} \to \bbR^{+} $ is its density, 
\begin{equation*}
\bbS(u_{\eps}) - p(\rho_{\eps}) =  2 \mu D(u_{\eps}) + (\lambda - \mu )\DIV(u_{\eps}) \bbI - \rho_{\eps}^{\gamma} \bbI ,
\end{equation*}
is the stress tensor,  $  D(u_{\eps})$ is the symmetric gradient, in other words $ 2  D(u_{\eps}) = \nabla u_{\eps} + ( \nabla u_{\eps} )^{T}  $ and $ \bbI $ is the two dimensional identity matrix. Moreover we assume $ \mu > 0 $,  $ \lambda \geq 0 $  and $ \gamma  > 1 $. Finally $ \rho_{\eps}^{in} \geq 0 $ is the initial density and $ q^{in}_{\eps} $, such that $ q^{in}_{\eps}(x) = 0 $ for $ x \in \{y \in \Omega $ such that $ \rho^{in}_{\eps} = 0  \} $, is the initial momentum.  

The above system have been widely studied in the past years and existence of finite energy weak solutions has been proved by Lions and Feireisl see \cite{Lions2} and \cite{NS}. 

In this paper we study the limit as $ \eps $ goes to zero for solutions of \eqref{CNS:equ}, in particular we show that under some mild assumption on $ \rho^{in}_{\eps}$ and $ u^{in}_{\eps} $ solutions $ (\rho_{\eps}, u_{\eps}) $ converge in an appropriate sense to a solution $ (\rho, u)$ of the \eqref{CNS:equ} with $ \calF_{\eps} $ replaced by $ \Omega $.

\section{Definition of weak solutions and main result}

In this section we recall the definition of weak solutions for the system \eqref{CNS:equ} and then we present the main result of the paper. 

In the following we denote by $ P $ the function $ P(\rho) = \rho^{\gamma}/ (\gamma - 1 ) $. For simplicity in the next definition we do do not write the small parameter $ \eps $.

\begin{Definition}
\label{DEF:CNS}

Let $ (\rho^{in}, q^{in}) $ be an initial data such that $ P(\rho^{in}) \in L^{1}(\calF) $ and $ |q^{in}|^2/\rho^{in} \in L^{1}(\calF) $. Then a triple $ ( \rho, u) $ is a weak solution of \eqref{CNS:equ} with initial datum $ (\rho^{in}, q^{in}) $ for any $ T > 0$ if

\begin{itemize}

\item $ \rho \in L^{\infty}(0,T; L^{1}(\calF)) $ such that $ \rho \geq 0 $ and $ P(\rho) \in L^{\infty}(0,T; L^{1}(\calF)) $.

\item $ u \in L^{2}(0,T;W^{1,2}_{0}(\calF)) $.

\item $ ( \rho, u) $ satisfies the transport equation $ \partial_{t}\rho + \DIV(\rho u) =  0 $ in both a distributional sense in $ [0,T)\times \R^{3} $ and in a renormalised sense where we extend $ \rho $ and $ u $ by zero in the exterior of $ [0,T] \times \calF $.

\item the momentum equation is satisfied in the weak sense
\begin{align*}
\int_{\calF} q^{in} \varphi(0,.)  & + \int_{0}^{T} \int_{\calF} (\rho u)\cdot \partial_t \varphi  + \int_{0}^{T} \int_{\calF}[\rho u \otimes u]: D \varphi + \rho^{\gamma} \DIV \varphi = \int_{0}^{T} \int_{\calF} \bbS u : D \varphi,
\end{align*}
for any $ \varphi \in C^{\infty}([0,T)\times \calF)$.

\item for a.e. $ \tau \in [0,T] $ the following energy equality holds
\begin{align*}
  \int_{\calF} \frac{1}{2}\rho |u|^2(\tau,.) + P(\rho(\tau,.)) \, dx + \int_{0}^{\tau}\int_{\calF} \mu | \nabla u|^2+ \lambda | \DIV u |^2 \, dx dt   \leq  \int_{\calF} \frac{1}{2} \frac{|q^{in}|^2}{\rho^{in}} +  P(\rho^{in}) dx.
\end{align*}

\end{itemize}

\end{Definition} 

We can now recall the existence result of weak solutions.

\begin{Theorem}

Let $ (\rho^{in}, q^{in}) $ be an initial data such that $ P(\rho^{in}) \in L^{1}(\calF) $ and $ |q^{in}|^2/\rho^{in} \in L^{1}(\calF) $. Then there exists a solutions $ ( \rho, u) $ of \eqref{CNS:equ}  in the sense of Definition \ref{DEF:CNS} with initial datum $ (\rho^{in}, q^{in}) $ for any $ T > 0 $.

\end{Theorem}

The proof is classical see for instance \cite{Lions2} or Section 7 of \cite{NS}. We are now able to state our main result.

\begin{Theorem}
\label{main:theo}

Let $ \gamma > 2 $ and let $ (\rho^{in}_{\eps}, q^{in}_{\eps}) $ be a sequence of initial data such that $ P(\rho^{in}_{\eps}) \in L^{1}(\calF_{\eps}) $ and $ |q^{in}_{\eps}|^2/\rho^{in}_{\eps} \in L^{1}(\calF_{\eps}) $ such that there exists $ \rho^{in} \in L^{\gamma}(\Omega) $ and $ q^{in} $ such that  $ q^{in} = 0 $ in $ \{ x $ such that $ \rho^{in } = 0 \}$  and $ |q^{in}|^2/\rho^{in} \in L^{1}(\calF_{\eps}) $ for which
\begin{itemize}

\item $ \rho^{in }_{\eps}  \lra \rho^{in } $ in $ L^{\gamma}(\Omega) $,

\item $ |q^{in}_{\eps}|^2/\rho^{in}_{\eps} \lra |q^{in}|^2/\rho^{in} $ in $ L^1(\Omega) $ .

\end{itemize}
Then up to subsequence there exists $ (\rho,u) $ such that
\begin{equation}
\rho_{\eps} \lra \rho \text{ in } C_w(0,T;L^{\gamma}(\Omega)) \quad \text{ and } \quad u_{\eps} \cv u  \text{ in } L^2(0,T;W^{1,2}_0(\Omega)).
\end{equation}
Moreover $ (\rho,u) $ satisfies \eqref{CNS:equ} in $ \Omega $ and with initial data $ (\rho^{in}, q^{in}) $, in the sense of Definition \ref{DEF:CNS}.

\end{Theorem}

Let us explain where we use the condition $ \gamma > 2 $.

\begin{Remark}
Although the existence of weak solutions to \eqref{CNS:equ} holds for $ \gamma > 1 $, in Theorem \ref{main:theo} we consider the case $ \gamma > 2 $. This restriction come from the fact that in dimension two to pass to the limit in the term
\begin{equation*}
\int_0^T \int_{\calF_{\eps}} \rho_{\eps}u_{\eps}\otimes u_{\eps} : D\Phi_{\eps}[\varphi],
\end{equation*}
we use that $ \rho_{\eps}$, $ u_{\eps} $ and $ D\Phi_{\eps}[\varphi] $  are uniformly bounded respectively in $ L^{\infty}(0,T;L^{\gamma}(\calF_{\eps})) $, in $ L^{2}(0,T;L^p(\calF_{\eps})) $ for any $ p   < +  \infty $ and $ L^{\infty}(0,T;L^{2}(\calF_{\eps})) $,  together with the condition 
\begin{equation*}
\frac{1}{\gamma}+ \frac{1}{\infty} + \frac{1}{\infty} + \frac{1}{2} \leq 1 \quad \text{ if and only if } \quad \gamma > 2.
\end{equation*}

\end{Remark}

In the remaining part of the paper we show Theorem \ref{main:theo}.

%%%%%%%%%%%%%%%%%%%%%%%%%%%%%%%%%%%%%%%%%%%%%%%%%%%%%%%%%%%%%%%%%%%%%%%%%%%%%%%%%%%%%%%%%%%%%%%%%%%%%%%%%%%%%%%%%%%%%%%%%%%%%%

\section{A priori estimates}
Let recall that by definition of weak solution to the system \eqref{CNS:equ}, any solution $ (\rho_{\eps}, u_{\eps}) $  satisfies the inequalities 
\begin{align*}
\|\rho_{\eps}(t,.)\|_{L^{\gamma}(\calF_{\eps})} =  \|\rho^{in}_{\eps}\|_{L^{\gamma}(\calF_{\eps})}
\end{align*}
and
\begin{align*}
  \int_{\calF_{\eps}} \frac{1}{2}\rho_{\eps} |u_{\eps}|^2(\tau,.) + P(\rho_{\eps}(\tau,.)) \, dx + \int_{0}^{\tau}\int_{\calF_{\eps}} \mu | \nabla u_{\eps}|^2+ & \lambda | \DIV u_{\eps} |^2 \, dx dt  \\ &  \leq  \int_{\calF} \frac{1}{2} \frac{|q^{in}_{\eps}|^2}{\rho^{in}_{\eps}} +  P(\rho^{in}_{\eps}) dx.
\end{align*}
By the hypothesis of Theorem \ref{main:theo}, the right hand side of the above inequalities are uniformly bounded in $ \eps $. In particular we deduce that
\begin{align}
\|\rho_{\eps}\|_{L^{\infty}(0,T;L^{\gamma}(\Omega))} \leq & C, \nonumber \\
\|\sqrt{\rho_{\eps}} u_{\eps} \|_{L^{\infty}(0,T;L^{2}(\Omega))} \leq & C,  \label{a:priori:est}\\
\|u_{\eps}\|_{L^{2}(0,T; W^{1,2}(\Omega))} \leq & C.\nonumber
\end{align}
Moreover we can show the following improved pressure estimates 
\begin{Lemma}
\label{press:est:lem}
Let $ \gamma > 2 $, under the hypothesis of Theorem \ref{main:theo}, for any $ \theta < \gamma -1 $ it holds
\begin{equation*}
\int_{0}^T \int_{\Omega \setminus B_{2 \eps}(0)} \rho^{\gamma + \theta }_{\eps} \leq  C, 
\end{equation*}
where $ C $ is independent of $ \eps $. 

\end{Lemma}

The proof of this estimates is classical so let us postpone the proof in the Appendix \ref{press:est:app}

\section{Some appropriate cut-off}

In this section we introduce some cut-off that has been considered also in \cite{He:2D} and \cite{ION}. These cut-off have the property that they optimized the $ L^{2} $ norm of the gradient and we denote them by $ \eta_{\eps, \alpha_\eps} $. The parameter $ \eps > 0 $ indicates that $ \eta_{\eps, \alpha_\eps} = 1 $ in the ball $ B_{\eps}(0) $ and $ \alpha_{\eps}  $ that the support of the  $ \eta_{\eps, \alpha_\eps} $ is contained in the ball of size $ \eps \alpha_{\eps } $. 

\begin{Proposition}
\label{prop:cut:off}
For any $ \eps > 0 $ and $ \alpha_{\eps} \geq 2 $, there exists a cut-off function $\eta_{\eps, \alpha_{\eps}}  \in C^{\infty}_c(B_{\eps \alpha_{\eps}}(0)) $ such that $ \eta_{\eps, \alpha_{\eps}}(x) = 1 $ for $ x \in B_{\eps}(0) $ ,  $ \| \eta_{\eps, \alpha_{\eps}} \|_{L^{\infty}} \leq 1 $ and the following bounds hold with constant $ C $ independent of $ \eps $ and $ \alpha_{\eps} $.
\begin{enumerate}

\item For  $ 1\leq q < +\infty $ $$\| \eta_{\eps, \alpha} \|_{L^{q}(\R^2)} + \| |x| \nabla \eta_{\eps, \alpha} \|_{L^q(\bbR^2)} \leq C(\eps \alpha_{\eps})^{2/q}  .$$

\item We have $$ \left\| \nabla \eta_{\eps, \alpha_{\eps}} \right\|_{L^{2}(\R^2)}^2  +  \left\| |x| \nabla^2 \eta_{\eps, \alpha_{\eps}} \right\|_{L^{2}(\R^2)}^2 \leq  \frac{C}{(\log \alpha_{\eps}) }. $$

\item For $ 1\leq q < 2 $, $$ \left\| \nabla \eta_{\eps, \alpha_{\eps}} \right\|_{L^{2}(\R^2)}^q + \left\| |x| \nabla^2 \eta_{\eps, \alpha_{\eps}} \right\|_{L^{q}(\R^2)}^2 \leq  \frac{C}{2-q}  \frac{(\eps \alpha_{\eps})^{2-q}}{(\log \alpha_{\eps})^{q} } . $$

\item For  $  2 <  q < +\infty $, , for $ i = 1, 2 $, $$ \left\| \nabla \eta_{\eps, \alpha_{\eps}} \right\|_{L^{q}(\R^2)}^q +  \left\| \nabla^2 \eta_{\eps, \alpha_{\eps}} x_i \right\|_{L^{q}(\R^2)}^q = \frac{C}{q-2}  \frac{\eps^{2-q}}{(\log \alpha_{\eps})^{q} }. $$

In particular if  $  \alpha_{\eps} \leq |\log(\eps)| $ and $ \alpha_{\eps} \lra +\infty $,
 $$ \left\| \nabla \eta_{\eps, \alpha_{\eps}} \right\|_{L^{q}(\R^2)} +  \left\| |x| \nabla^2 \eta_{\eps, \alpha_{\eps}} \right\|_{L^{2}(\R^2)} \lra 0 \quad \text{ for } 1 \leq q \leq 2 $$
and 
$$  \eps \alpha_{\eps} \left\| \nabla \eta_{\eps, \alpha_{\eps}} \right\|_{L^{q}(\R^2)}, \eps \alpha_{\eps} \left\| \nabla^2 \eta_{\eps, \alpha_{\eps}} x_i \right\|_{L^{q}(\R^2)} \lra 0 \quad \text{ for }  2 <  q < +\infty.   $$

\end{enumerate}

\end{Proposition}

The proof of the above proposition is a straight-forward extension of Lemma 3 of  \cite{He:2D}, so let us postpone the proof in Appendix \ref{app:Jiao}.

Under the assumption $  \alpha_{\eps} \leq |\log(\eps)| $ and $ \alpha_{\eps} \lra +\infty $, we denote $ 1 - \eta_{2\eps, \alpha_{2\eps}} = \mathfrak{n}_{\eps}$.

Let now present another useful estimate. 
For a function $ \varphi \in L^{1}(\Omega)$, denote by
\begin{equation*}
\Phi_{\eps}[\varphi] =   \fn_{\eps} \varphi + \nabla^{\perp} \fn_{\eps} x^{\perp} \cdot \varphi
\end{equation*}
and by
\begin{equation*}
\Phi_{\eps}^0[\varphi] =   (1-\fn_{\eps}) \langle \varphi \rangle_{\eps}(0) - \nabla^{\perp} \fn_{\eps} x^{\perp} \cdot \langle \varphi \rangle_{\eps}(0) 
\end{equation*}
where
\begin{equation*}
\langle \varphi \rangle_{\eps}(0) = \frac{1}{|B_{\eps \alpha_{\eps}}(0)|}\int_{|B_{\eps \alpha_{\eps}}(0)|} \varphi.
\end{equation*}
The following holds. 

\begin{Lemma}
\label{lem:est:tf}
Let $ p, q \in [1,+\infty]$.Then there exist a constants $ c_{p,q}(\eps) $ such that $ c_{p,q}(\eps) \lra 0 $ as $ \eps \lra 0 $  such that for any vector field $ \varphi : \Omega \lra \bbR^2 $, it holds 
\begin{equation*}
\|  \Phi_{\eps}[\varphi] - \varphi \|_{L^{p}(\Omega)} \leq  c_{p,q}(\eps) \|\varphi \|_{L^{q}(\Omega)} \quad \text{and } \quad  \|  \Phi_{\eps}[\varphi] - \fn_{\eps} \varphi \|_{L^{p}(\Omega)} \leq  c_{p,q}(\eps) \|\varphi \|_{L^{q}(\Omega)}
\end{equation*}
for $ p < q < \infty $.

\begin{equation*}
\|  \nabla \Phi_{\eps}[\varphi] - \fn_{\eps} \nabla \varphi  \|_{L^{p}(\Omega)} \leq  c_{p,q}(\eps) \|\varphi \|_{W^{1,q}(\Omega) }  
\end{equation*} 
for $ p \leq 2 $ and $ q > 2 $.
Finally 
\begin{equation}
\label{div:est:phi:0}
\|\DIV(\Phi_{\eps}[\varphi]) - \fn_{\eps} \DIV(\varphi)\|_{L^{p}(\Omega)} \leq  c_{p,q}(\eps) \|\varphi\|_{L^{q}(\Omega)}
\end{equation}
for $ p < q < \infty $.

\end{Lemma}

In the following we always omit the dependence on $ p $ and $ q $ for $  c_{p,q}(\eps) $ and we write $  c_{p,q}(\eps) = c(\eps ) $. 

\begin{proof}

The proof of these inequalities follows from the definition of $ \Phi_{\eps}[\varphi] $ and Proposition \ref{prop:cut:off}. The most interesting one is \eqref{div:est:phi:0}, so we will prove it. First  of all notice that $ \DIV(\Phi_{\eps}^0[\varphi]) = 0 $, in fact 
\begin{align*}
\DIV(\Phi_{\eps}^0[\varphi]) =  (1-\fn_{\eps}) \langle \varphi \rangle_{\eps}(0) - \nabla^{\perp} \fn_{\eps} x^{\perp} \cdot \langle \varphi \rangle_{\eps}(0)  = \DIV(  \nabla^{\perp}((1-\fn_{\eps}) x^{\perp} \cdot \langle \varphi \rangle_{\eps}(0) )) = 0.
\end{align*} 
Then 
\begin{align*}
\DIV( \Phi_{\eps}[\varphi]) - \fn_{\eps} \DIV(\varphi) = & \, \DIV( \Phi_{\eps}[\varphi]) + \DIV(\Phi_{\eps}^0[\varphi])  - \fn_{\eps}\DIV(\varphi) \\
= & \, \nabla \fn_{\eps}(\varphi -  \langle \varphi \rangle_{\eps}(0) ) + \nabla^{\perp}  \fn_{\eps} \otimes x^{\perp} : \nabla \varphi \\ 
& \, + \nabla^{\perp} \fn_{\eps} \otimes (\varphi -  \langle \varphi \rangle_{\eps}(0)  ) : \nabla x^{\perp}.
\end{align*}
Using the above equality, we estimate for $ 1/s = 1/p - 1/q $ 
\begin{align*}
\| \DIV( \Phi_{\eps}[\varphi])  & \, - \fn_{\eps} \DIV(\varphi) \|_{L^{p}(\Omega)} \leq  \| \nabla \fn_{\eps} \|_{L^{s}(\Omega)} \| \varphi -  \langle \varphi \rangle_{\eps}(0) \|_{L^{q}(B_{2\eps \alpha_{2\eps}(0) })}  \\
& \, + \| \nabla^{\perp}  \fn_{\eps} \otimes x^{\perp} \|_{L^{s}(\Omega)}\|  \nabla \varphi \|_{L^{q}(\Omega) } + \|  \nabla^{\perp} \fn_{\eps} \|_{L^{s}(\Omega)} \| \varphi -  \langle \varphi \rangle_{\eps}(0)  \|_{L^{q}(B_{2\eps \alpha_{2\eps}(0) }) } \\
\leq & \, \eps \alpha_{\eps}  C \| \nabla \fn_{\eps} \|_{L^{s}(\Omega)} \| \varphi  \|_{W^{1,q}(\Omega)} + C  \| \nabla^{\perp}  \fn_{\eps} \otimes x^{\perp} \|_{L^{s}(\Omega)}\|  \nabla \varphi \|_{L^{q}(\Omega) } \\
& \, +  \eps \alpha_{\eps}  C \|  \nabla^{\perp} \fn_{\eps} \|_{L^{s}(\Omega)}\|   \varphi \|_{W^{1,q}(\Omega)} \\
\leq & c(\eps) \|\varphi \|_{W^{1,q}(\Omega)},
\end{align*}
where we use that $ p < q $, the Poincar\'e inequality and Proposition \ref{prop:cut:off}.
\end{proof}

\section{Pass to the limit in the weak formulation}

Using the estimates from  \eqref{a:priori:est}, Lemma \ref{press:est:lem} and the fact that $ (\rho_{\eps}, u_{\eps}) $ are solutions to the system \eqref{CNS:equ} the following convergences hold.  

\begin{Lemma}
 \label{Lemma:convergence}
Under the hypothesis of Theorem \ref{main:theo}, we have after passing to subsequence that
\begin{align*}
\rho_{\eps} \cv \rho & \quad && \text{ in } L^{2\gamma-1}([0,T] \times \Omega)  \\
\rho_{\eps} \lra \rho & \quad  && \text{ in } C^0_{w}([0,T); L^{\gamma}(\Omega)) \\
u_{\eps} \cv u & \quad && \text{ in } L^{2}(0,T; H^1_{0}(\Omega)) \\
\fn_{\eps}\rho_{\eps} u_{\eps} +\rho_{\eps}u_{\eps} \cdot \nabla^{\perp} \fn_{\eps}x^{\perp}  \lra  \rho u  & \quad  && \text{ in } C^0_{w}([0,T ); L^{2\gamma/(\gamma+1)}(\Omega)) \\
\left(\fn_{\eps}\rho_{\eps} u_{\eps} +\rho_{\eps}u_{\eps} \cdot \nabla^{\perp} \fn_{\eps} x^{\perp}\right) \otimes u_{\eps} \lra  \rho u \otimes u  & \quad  && \text{ in } \mathcal{D}'((0,T) \times \Omega) \\
\mathds{1}_{B_{2\eps}(0)}\rho_{\eps}^{\gamma} \cv \overline{\rho^{\gamma}} & \quad && \text{ in } L^{(2\gamma-1)/\gamma}([0,T] \times \Omega)
\end{align*}
where $ t \in (0,T)$.

\end{Lemma}

\begin{proof}

Using \eqref{a:priori:est}, Lemma \ref{press:est:lem} and the fact that $ (\rho_{\eps}, u_{\eps}) $ are solutions to the system \eqref{CNS:equ}, it is easy to deduce all the convergence except the fourth one. By  \eqref{a:priori:est} we already know that
\begin{equation*}
\|\rho_{\eps} u_{\eps}\|_{L^{\infty}(0,T;L^{2\gamma/(\gamma+1)}(\F_{\eps}))} \leq  \|\sqrt{\rho_{\eps}}\|_{L^{\infty}(0,T;L^{2\gamma}(\F_{\eps}))}\|\sqrt{\rho_{\eps}}u_{\eps}\|_{L^{\infty}(0,T;L^{2}(\F_{\eps}))} \leq C.
\end{equation*}
Moreover using that $ |\fn_{\eps}| $ and $ |\nabla^{\perp}\fn_{\eps}\otimes x^{\perp}| $ are bounded, we deduce that up to subsequence
$$ \fn_{\eps}\rho_{\eps} u_{\eps} +\rho_{\eps}u_{\eps} \cdot \nabla^{\perp} \fn_{\eps} x^{\perp}  \cvwstar  \rho u   \quad   \text{ in }L^{\infty}([0,T ); L^{2\gamma/(\gamma+1)}(\Omega)).  $$
To show the strong convergence in time, it is enough to prove that $  \fn_{\eps}\rho_{\eps} u_{\eps} +\rho_{\eps}u_{\eps} \cdot \nabla^{\perp} \fn_{\eps}x^{\perp} $ is continuous an equicontinuous in some $ H^{-s}(\Omega) $ for some $ s $ big enough and to apply Appendix C of \cite{Lions1}. To do that we will apply the following lemma.

\begin{Lemma}
\label{lem:equi:con}
Let $ H $ an Hilbert space and let $ f_{n}: (0,T) \lra H $ a sequence of functions. If $ \partial_t f_{n} = g_{n}^1 + g_{n}^2 $ where 
\begin{itemize}

\item $ \|g_{n}^1\|_{L^{p}(0,T;H)} \leq C $ with $ C $ independent of $ n $ and $ p > 1 $,

\item $ \lim_{n \to +\infty} \|g_n^2\|_{L^{1}(0,T;H)} =0 $.

\end{itemize}
Then the functions $ f_{n} $ are continuous and equicontinuous.
\end{Lemma}

For $ \varphi \in C^{\infty}((0,T)\times \Omega ) $, we notice that 
\begin{align}
\label{iner:time:der}
\int_{0}^T \int_{\Omega} \left(  \fn_{\eps}\rho_{\eps} u_{\eps} +\rho_{\eps}u_{\eps} \cdot \nabla^{\perp} \fn_{\eps}x^{\perp} \right)\cdot \partial_t \varphi = & \, \int_{0}^T \int_{\Omega} \rho_{\eps} u_{\eps} \cdot \partial_ t \left(  \fn_{\eps}\varphi + \nabla^{\perp} \fn_{\eps}x^{\perp} \cdot \varphi \right) \\
= & \, \int_{0}^T \int_{\Omega} \rho_{\eps} u_{\eps} \cdot \partial_t \Phi_{\eps}[\varphi]. \nonumber
\end{align}
We can now use the momentum equation of \eqref{CNS:equ} tested with  $ \Phi_{\eps}[\varphi] $ to deduce
\begin{align}
\label{der:in:time}
\int_{0}^T \int_{\Omega}   \rho_{\eps} u_{\eps} \cdot \partial_ t   \Phi_{\eps}[\varphi] = & \,  - \int_0^T \int_ {\Omega} \rho_{\eps} u_{\eps}\otimes u_{\eps}:D\Phi_{\eps}[\varphi] \\ & \, +  \int_0^T\int_{\Omega} \mathbb{S}u : D\Phi_{\eps}[\varphi] - \int_0^T \int_{\Omega} \nabla \rho^{\gamma} \DIV(\Phi_{\eps}[\varphi]). \nonumber
\end{align}
We now bound the terms on the right hand side separately.
Notice that 
\begin{align*}
 \int_0^T \int_ {\Omega} \rho_{\eps} u_{\eps}\otimes u_{\eps}:D\Phi_{\eps}[\varphi]  = & \, \int_0^T \int_ {\Omega} \rho_{\eps} u_{\eps}\otimes u_{\eps}:\fn_{\eps} D\varphi  \\ & \,  + \int_0^T \int_ {\Omega} \rho_{\eps} u_{\eps}\otimes u_{\eps}:\left(D\Phi_{\eps}[\varphi] - \fn_{\eps} D\varphi\right).
\end{align*}
Let recall that in dimension two $ W^{1,2} \subset L^{p} $ for any $ p < + \infty $ and that $ W^{1,2} \not\subset L^{\infty} $, in particular $ \|f\|_{L^p} \leq \|f\|_{W^{1,2}} $ for any $ p < + \infty $. In the following we denote by $ \|f\|_{L^{\infty^-}} $ the norm  $ \|f\|_{L^{p}} $ for $ p $ big enough. 
We have
\begin{equation*}
\left| \int_0^T \int_ {\Omega} \rho_{\eps} u_{\eps}\otimes u_{\eps}:\fn_{\eps} D\varphi  \right| \leq \| \rho_{\eps} u_{\eps} \|_{L^{\infty}(0,T:L^{2\gamma/(\gamma+1)}(\F_{\eps}))}\|u_{\eps}\|_{L^2(0,T;L^{\infty^-}(\F_{\eps}))}\|D\varphi\|_{L^{2}(0,T; L^{q}(\Omega))},
\end{equation*}
for $ q > 2\gamma/(\gamma-1)$. Moreover
\begin{align}
& \left| \int_0^T \int_ {\Omega} \rho_{\eps} u_{\eps}\otimes u_{\eps}:\left(D\Phi_{\eps}[\varphi] - \fn_{\eps} D\varphi\right) \right|  \nonumber \\  & \quad \quad \quad \leq \|\rho_{\eps}\|_{L^{\infty}(0,T;L^{\gamma}(\Omega))} \| u_{\eps}\|_{L^{2}(0,T;L^{\infty^-}(\Omega))}^{2} \|D\Phi_{\eps}[\varphi] - \fn_{\eps} D\varphi\|_{L^{1}(0,T; L^{\tilde{q}}(\Omega))} \label{conv:err:nonlin} \\ 
 & \quad \quad \quad \leq c(\eps) \|\rho_{\eps}\|_{L^{\infty}(0,T;L^{\gamma}(\Omega))} \| u_{\eps}\|_{L^{2}(0,T;L^{\infty^-}(\Omega))}^{2} \|\varphi\|_{L^{1}(0,T; W^{1,q}(\Omega))} \nonumber 
\end{align}
where $ 2 > q > \tilde{q} >  \gamma/ (\gamma -1 ) $ and $ c(\eps) \lra 0 $ as $ \eps \lra 0 $. In the last inequality we used Lemma \ref{lem:est:tf}. Let now move to the second term of right hand side of \eqref{der:in:time}. As before 
\begin{align*}
\left|\int_0^T\int_{\Omega} \mathbb{S}u :  D\Phi_{\eps}[\varphi] \right| \leq & \,  \left|\int_0^T\int_{\Omega} \mathbb{S}u : \fn_{\eps} D\varphi \right| + \left|\int_0^T\int_{\Omega} \mathbb{S}u : \left(D\Phi_{\eps} - \fn_{\eps} D\varphi\right) \right| \\
\leq & (1+c(\eps)) \|u_{\eps}\|_{L^{2}(0,T;W^{1,2}(\Omega)) } \|\varphi\|_{L^{2}(0,T; W^{1,2}(\Omega))} 
\end{align*}
where $ c(\eps) \lra 0 $ as $ \eps \lra 0 $. We are left with the last term of \eqref{der:in:time}. Using that $ \Phi_{\eps}[\varphi] = 0 $ in $ B_{2\eps}(0) $, we rewrite 
 \begin{align}
\int_{0}^T\int_{\Omega} \rho_{\eps}^{\gamma} \DIV(\Phi_{\eps}[\varphi])  = & \, \int_{0}^T\int_{\Omega}\mathds{1}_{B_{2\eps}(0)} \rho_{\eps}^{\gamma} \fn_{\eps} \DIV(\varphi) + \int_{0}^T\int_{\Omega} \mathds{1}_{B_{2\eps}(0)} \rho_{\eps}^{\gamma} \left(\DIV(  \Phi_{\eps}[\varphi]  )  - \fn_{\eps}\DIV(\varphi) \right) \label{secon:term:sei}
\end{align}
Notice that 
\begin{equation*}
\left| \int_{0}^T\int_{\Omega}\mathds{1}_{B_{2\eps}(0)} \rho_{\eps}^{\gamma} \DIV(\varphi) \right|  \leq \| \mathds{1}_{B_{2\eps}(0)} \rho^{\gamma} \|_{L^{p}((0,T) \times \Omega )} \|\DIV(\varphi)\|_{L^{q}((0,T) \times \Omega )}
\end{equation*}
is uniformly bounded for any $ 1/p + 1/q = 1 $ and $ p < (2\gamma-1)/\gamma $.  

The same estimate holds for the second term of \eqref{secon:term:sei}, moreover from Lemma \ref{lem:est:tf}, we have
\begin{align}
& \left| \int_{0}^T\int_{\Omega} \mathds{1}_{B_{2\eps}(0)} \rho_{\eps}^{\gamma} \left(\DIV(  \Phi_{\eps}[\varphi]  )  - \fn_{\eps}\DIV(\varphi) \right) \right| \nonumber \\ & \quad \quad \quad \leq   \| \mathds{1}_{B_{2\eps}(0)} \rho^{\gamma} \|_{L^{p}((0,T) \times \Omega )} \| \DIV(  \Phi_{\eps}[\varphi]  )  - \fn_{\eps}\DIV(\varphi) \|_{L^{q}((0,T) \times \Omega )} \label{happy} \\ & \quad \quad \quad  \leq  c(\eps) \| \mathds{1}_{B_{2\eps}(0)} \rho^{\gamma} \|_{L^{p}((0,T) \times \Omega )} \| \varphi \|_{L^{q}(0,T;W^{1,\tilde{q}}(\Omega))} \nonumber
\end{align}
with $ \tilde{q} > \max\{ 2, q \}$ and with $ c(\eps) \lra 0 $ as  $ \eps \lra 0 $.

We can now apply Lemma \ref{lem:equi:con} and deduce 
\begin{equation*}
\fn_{\eps}\rho_{\eps} u_{\eps} +\rho_{\eps}u_{\eps} \cdot \nabla^{\perp} \fn_{\eps}x^{\perp}  \lra  \rho u  \quad   \text{ in } C^0_{w}([0,T ); L^{2\gamma/(\gamma+1)}(\Omega)).
\end{equation*}

\end{proof}

We will now pass to the limit in the weak formulation satisfied by $ \rho_{\eps} $ and $ u_{\eps} $. First of all notice that it is easy to pass to the limit in the transport equation satisfied by the density, so let us concentrate on the momentum equation. For $ \varphi \in C^{\infty}_{c}([0,T) \times \Omega) $ we test  the weak formulation of the momentum equation satisfied by $\rho_{\eps}$, $ u_{\eps}$ with $ \Phi_{\eps}[\varphi] = \fn_{\eps} \varphi + \nabla^{\perp} \fn_{\eps} x^{\perp} \cdot \varphi $. We deduce that 
\begin{align*}
\int_{\calF_{\eps}} q^{in}_{\eps} \Phi_{\eps}[\varphi](0,.)  + \int_{0}^{T} \int_{\calF_{\eps}} (\rho_{\eps} u_{\eps})\cdot \partial_t \Phi_{\eps}[\varphi]  + \int_{0}^{T} \int_{\calF_{\eps}}[\rho_{\eps} u_{\eps} \otimes u_{\eps}]: & \,  D \Phi_{\eps}[\varphi] + \rho^{\gamma}_{\eps} \DIV \Phi_{\eps}[\varphi]\\  =  & \, \int_{0}^{T} \int_{\calF_{\eps}} \bbS u_{\eps} : D \Phi_{\eps}[\varphi].
\end{align*}
We will pass to the limit in $ \eps $ in any term separately. First of all notice that 
$$ \Phi_{\eps}[\varphi](0,.) = \fn_{\eps} \varphi(0,.) + \nabla^{\perp} \fn_{\eps} x^{\perp} \cdot \varphi(0,.) \lra \varphi(0,.) \quad \text{ in } L^{q}(\Omega) $$
for any $ q < +\infty $ by dominate convergence. We deduce that 
\begin{equation*}
\int_{\calF_{\eps}} q^{in}_{\eps} \Phi_{\eps}[\varphi](0,.) = \int_{\calF_{\eps}}\frac{ q^{in}_{\eps}}{\sqrt{\rho_{\eps}^{in}}} \sqrt{\rho_{\eps}^{in}}\Phi_{\eps}(0,.)  \lra \int_{\Omega}\frac{ q^{in}}{\sqrt{\rho^{in}}} \sqrt{\rho^{in}}\varphi(0,.) = \int_{\Omega} q^{in} \varphi(0,.) 
\end{equation*}
where we used $q^{in}_{\eps}/ \sqrt{\rho_{\eps}^{in}}  \lra q^{in}/ \sqrt{\rho^{in}}  $ in $ L^2(\Omega)$ and  $  \sqrt{\rho_{\eps}^{in}}  \lra  \sqrt{\rho^{in}}  $ in $ L^{2\gamma}(\Omega)$ . Using \eqref{iner:time:der}, we notice that
\begin{align*}
\int_{0}^{T} \int_{\calF_{\eps}} (\rho_{\eps} u_{\eps})\cdot \partial_t \Phi_{\eps}[\varphi] = \int_{0}^{T} \int_{\calF_{\eps}} (\fn_{\eps}\rho_{\eps} u_{\eps} +\rho_{\eps}u_{\eps} \cdot \nabla^{\perp} \fn_{\eps}x^{\perp} )\cdot \partial_t \varphi \lra \int_{0}^{T} \int_{\Omega} \rho u \cdot \partial_t \varphi
\end{align*}
where we use the convergence from Lemma \ref{Lemma:convergence}. For the next term let rewrite
\begin{align}
\int_{0}^{T} \int_{\calF_{\eps}}[\rho_{\eps} u_{\eps} \otimes u_{\eps}]: D \Phi_{\eps}[\varphi] = & \, \int_{0}^{T} \int_{\calF_{\eps}}\left[\left(\fn_{\eps}\rho_{\eps} u_{\eps} +\rho_{\eps}u_{\eps} \cdot \nabla^{\perp} \fn_{\eps} x^{\perp}\right) \otimes u_{\eps}\right]: D \varphi \nonumber \\
& \, + \int_{0}^{T} \int_{\calF_{\eps}}[\rho_{\eps} u_{\eps} \otimes u_{\eps}]: \left(D \Phi_{\eps}[\varphi] - \fn_{\eps} D\varphi\right) \label{12} \\
& \, - \int_{0}^{T} \int_{\calF_{\eps}}\left[\rho_{\eps}u_{\eps} \cdot \nabla^{\perp} \fn_{\eps} x^{\perp} \otimes u_{\eps}\right]: D \varphi. \nonumber
\end{align}
Notice that 
\begin{equation*}
 \int_{0}^{T} \int_{\calF_{\eps}}\left[\left(\fn_{\eps}\rho_{\eps} u_{\eps} +\rho_{\eps}u_{\eps} \cdot \nabla^{\perp} \fn_{\eps} x^{\perp}\right) \otimes u_{\eps}\right]: D \varphi  \lra \int_{0}^{T} \int_{\Omega} \rho u \otimes u : D \varphi,
\end{equation*}
due to Lemma \ref{Lemma:convergence}. Moreover the second term of the right hand side of \eqref{12} converges to zero due to  \eqref{conv:err:nonlin}. Finally the last term of  \eqref{12} 
\begin{align*}
\Bigg| \int_{0}^{T} \int_{\calF_{\eps}} & \left[\rho_{\eps}u_{\eps} \cdot \nabla^{\perp} \fn_{\eps} x^{\perp} \otimes u_{\eps}\right]: D \varphi \Bigg|\\  \leq & \, \|\rho_{\eps}\|_{L^{\infty}(0,T;L^{\gamma}(\Omega))} \|u_{\eps}\|_{L^{2}(0,T;L^{\infty^-}(\Omega))}^{2} \|\nabla^{\perp} \fn_{\eps} x^{\perp} \|_{L^2(\Omega)} \\ & \,  \lra 0,
\end{align*}
where we use Proposition \ref{prop:cut:off}. We deduce 
\begin{equation*}
\int_{0}^{T} \int_{\calF_{\eps}}[\rho_{\eps} u_{\eps} \otimes u_{\eps}]: D \Phi_{\eps}[\varphi] \lra \int_{0}^{T} \int_{\Omega} \rho u \otimes u : D \varphi.
\end{equation*}
The next term is
\begin{align*}
\int_{0}^{T} \int_{\calF_{\eps}}  \rho^{\gamma}_{\eps} \DIV \Phi_{\eps}[\varphi] = & \, \int_{0}^{T} \int_{\calF_{\eps}} \mathds{1}_{B_{2\eps}(0)} \rho^{\gamma}_{\eps} \DIV \varphi +  \int_{0}^{T} \int_{\calF_{\eps}} \mathds{1}_{B_{2\eps}(0)} \rho^{\gamma}_{\eps}\left( \DIV( \Phi_{\eps}[\varphi] - \fn_{\eps} \DIV(\varphi)\right)  \\
&  \, \lra  \int_{0}^{T} \int_{\Omega}  \overline{\rho^{\gamma}} \DIV \varphi,
\end{align*}
where we used Lemma \ref{Lemma:convergence} for the convergence of the first term and \eqref{happy} for the second one. Finally
\begin{align*}
\int_{0}^{T} \int_{\calF_{\eps}} \bbS u_{\eps} : D \Phi_{\eps} = & \, \int_{0}^{T} \int_{\calF_{\eps}} \bbS u_{\eps} : \fn_{\eps} D \varphi + \int_{0}^{T} \int_{\calF_{\eps}} \bbS u_{\eps} :( D \Phi_{\eps} - D \fn_{\eps} D \varphi) \\ & \,  \lra \int_{0}^{T} \int_{\Omega} \bbS u : D \varphi
\end{align*}
where we used the weak convergence of $ u_{\eps} $ from Lemma \ref{Lemma:convergence} and the strong convergence of $ \fn_{\eps} D \varphi \lra D\varphi $ in $ L^{2} $. The second term converge to zero from Lemma \eqref{lem:est:tf}. 

Putting all this convergence together we deduce that $ \rho $ and $ u $ satisfy
\begin{align*}
\int_{\Omega} q^{in} \varphi(0,.)  + \int_{0}^{T} \int_{\Omega} \rho u \cdot \partial_t \varphi  + \int_{0}^{T} \int_{\Omega}[\rho u \otimes u]: D \varphi +  \overline{ \rho^{\gamma}} \DIV \varphi  =  \int_{0}^{T} \int_{\Omega} \bbS u : D \varphi.
\end{align*}

It now remains to show that $ \overline{\rho^{\gamma}} = \rho^{\gamma} $. We will show this in the next section.

\section{Identification of the pressure}

We now show that $ \overline{\rho^{\gamma}} = \rho^{\gamma} $ to do that  we follow the strategy introduced by Lions in \cite{Lions2}. Let recall that we are in the case $ \gamma > 2 $ and dimension two, it is then enough to show the following lemma.

\begin{Lemma}
For any $ \psi \in C^{\infty}_{c}(\Omega) $, it holds
\begin{equation*}
\lim_{\eps \lra 0} \int_{0}^T \int_{\Omega} \psi^{2}\fn_{\eps}\left( \rho_{\eps}^{\gamma} - (2\mu+\gamma) \DIV(u_{\eps})      \right) \rho_{\eps} = \int_{0}^T \int_{\Omega} \psi^2\left(\overline{\rho^{\gamma}} -(2\mu+\lambda)\DIV(u)\right) \rho
\end{equation*}
up to subsequence.
\end{Lemma}

\begin{proof}
Consider $ \phi_{\eps} = \psi \Phi_{\eps}[\nabla \Delta^{-1} [\psi \rho_{\eps}]]= \psi \fn_{\eps} \nabla \Delta^{-1} [\psi \rho_{\eps}] + \psi \nabla^{\perp}\fn_{\eps}x^{\perp} \cdot \nabla \Delta^{-1} [\psi \rho_{\eps}] $ and $ \phi = \psi  \nabla \Delta^{-1} [\psi \rho] $. From the a priori estimates on the solutions $ \rho_{\eps} $, $ u_{\eps} $, we notice that $ \nabla \Delta^{-1} [\psi \rho_{\eps}] $ is uniformly bounded in $ L^{\infty}(0,T;W^{1,\gamma}(\Omega))  $ and $ \partial_{t }  \nabla \Delta^{-1} [\psi \rho_{\eps}] = -  \nabla \Delta^{-1} [\psi\DIV( \rho_{\eps}u_{\eps})] $ is uniformly bounded is some $ L^p $ spaces. We can now test the weak formulation satisfied by  $ \rho_{\eps} $, $ u_{\eps} $ by $ \phi_{\eps} $ and the one of $ \rho $, $ u $ by $ \phi $. Using the convergence of initial data  
we deduce
\begin{align}
\lim_{\eps \to 0 }   \int_{0}^{T}\int_{\Omega} & \rho_{\eps} u_{\eps} \cdot \partial_t \phi_{\eps} + \rho_{\eps} u_{\eps} \otimes u_{\eps} : \nabla \phi_{\eps} + \rho_{\eps}^{\gamma} \DIV(\phi_{\eps})- \mathbb{S}u_{\eps}: \nabla \phi_{\eps} \label{5:5}\\
= & \int_{0}^{T}\int_{\Omega} \rho u \cdot \phi + \rho u \otimes u : \nabla \phi + \overline{\rho^{\gamma}} \DIV(\phi)- \mathbb{S}u: \nabla \phi. \nonumber
\end{align}
We will now rewrite in an appropriate way the above equality. Notice that
\begin{align*}
\partial_t \phi_{\eps} = \psi \Phi_{\eps}[\nabla \Delta^{-1}(\psi \partial_t \rho_{\eps})] = - \psi \Phi_{\eps}[\nabla \Delta^{-1}(\DIV(\psi \rho_{\eps} u_{\eps})] + \psi \Phi_{\eps} [\nabla \Delta^{-1}(\nabla(\psi) \rho_{\eps} u_{\eps}]
\end{align*}
We deduce that
\begin{align*}
 \int_{0}^{T}\int_{\Omega}  \rho_{\eps} u_{\eps} \cdot \partial_t \phi_{\eps} = & \,   - \int_{0}^{T}\int_{\Omega} \psi  \fn_{\eps} \rho_{\eps} u_{\eps} \cdot \nabla \Delta^{-1}(\DIV(\psi \rho_{\eps} u_{\eps})) \\
 & \,  - \int_{0}^{T}\int_{\Omega} \psi  \fn_{\eps} \rho_{\eps} u_{\eps} \cdot (\Phi_{\eps} - \text{Id})[\nabla \Delta^{-1}(\DIV(\psi \rho_{\eps} u_{\eps}))] \\
 & \, + \int_{0}^{T}\int_{\Omega}    \rho_{\eps} u_{\eps} \cdot \psi \Phi_{\eps}[\nabla \Delta^{-1}(\nabla(\psi) \rho_{\eps} u_{\eps})].
\end{align*}
The second term converges to zero as $ \eps $ goes to zero due to Lemma \ref{lem:est:tf} in fact it contains $ \Phi_{\eps} - \text{Id} $. Let rewrite the third term
\begin{align}
\int_{0}^{T}\int_{\Omega}    \rho_{\eps} u_{\eps} \cdot \psi \Phi_{\eps} &\,[\nabla \Delta^{-1}(\nabla(\psi) \rho_{\eps} u_{\eps})] =    \int_{0}^{T}\int_{\Omega}    \psi \Phi_{\eps}[\rho_{\eps} u_{\eps}]\cdot \nabla \Delta^{-1}(\nabla(\psi) \Phi_{\eps}[\rho_{\eps} u_{\eps})] \nonumber  \\
 & \, + \int_{0}^{T}\int_{\Omega}    \psi \Phi_{\eps}[\rho_{\eps} u_{\eps}] \cdot \nabla \Delta^{-1}(\nabla(\psi) (\text{Id}-\Phi_{\eps})[\rho_{\eps} u_{\eps}] \label{strategy:1} \\
  & \, + \int_{0}^{T}\int_{\Omega}    \psi (\text{Id}-\Phi_{\eps})[\rho_{\eps} u_{\eps}] \cdot \nabla \Delta^{-1}(\nabla(\psi) \rho_{\eps} u_{\eps})  \nonumber \\
 & \, + \int_{0}^{T}\int_{\Omega}    \rho_{\eps} u_{\eps} \cdot \psi(\Phi_{\eps}-\text{Id})[\nabla \Delta^{-1}(\nabla(\psi) \rho_{\eps} u_{\eps})]. \nonumber 
\end{align}
The first three terms  converge to zero as $ \eps $ goes to zero due to Lemma \ref{lem:est:tf}. To tackle the last one, notice that $ \nabla \Delta^{-1}(\nabla(\psi) \Phi_{\eps}[\rho_{\eps} u_{\eps}])  $ converges strongly to $ \nabla \Delta^{-1}(\nabla(\psi) \rho u )  $ in $ C^{0}(0,T;L^{q}(\Omega)) $ for any $ q < 2 \gamma $ and $  \Phi_{\eps}[\rho_{\eps} u_{\eps}] $ converges to $ \rho u $ in $ C_{w}(0,T; L^{2\gamma/(\gamma+1)}(\Omega)) $, we deduce that 
\begin{align}
 \int_{0}^{T}\int_{\Omega}  \rho_{\eps} u_{\eps} \cdot \partial_t \phi_{\eps} = &  \, - \int_{0}^{T}\int_{\Omega} \psi  \fn_{\eps} \rho_{\eps} u_{\eps} \cdot \nabla \Delta^{-1}(\DIV(\psi \rho_{\eps} u_{\eps})) \label{1:1} \\ & \, + \int_{0}^{T}\int_{\Omega}    \psi \rho u \cdot \nabla \Delta^{-1}(\nabla(\psi) \rho u ))  +  c(\eps), \nonumber
\end{align}
with $ c(\eps) \lra 0 $ as $ \eps \lra 0 $.
Similarly we have 
\begin{align*}
\int_{0}^{T}\int_{\Omega} \rho_{\eps} u_{\eps} \otimes u_{\eps} : \nabla \phi_{\eps} = & \, \int_{0}^{T}\int_{\Omega} \rho_{\eps} u_{\eps} \otimes u_{\eps} : (\nabla \psi \otimes \Phi_{\eps}[\nabla \Delta^{-1}(\psi \rho_{\eps}))] \\
&\, +  \int_{0}^{T}\int_{\Omega}\fn_{\eps} \psi \rho_{\eps} u_{\eps} \otimes u_{\eps} : \nabla (\nabla \Delta^{-1}(\psi \rho_{\eps}))  \\
&\, +  \int_{0}^{T}\int_{\Omega}  \psi \rho_{\eps} u_{\eps} \otimes u_{\eps} : (\nabla \Phi_{\eps} -\fn_{\eps} \nabla)[\nabla \Delta^{-1}(\psi \rho_{\eps})]
\end{align*}
Following the strategy of $ \eqref{strategy:1} $ we deduce that the first term of the right hand side converges to 
\begin{equation*}
\int_{0}^{T}\int_{\Omega} \rho u \otimes u : (\nabla \psi \otimes \nabla \Delta^{-1}(\psi \rho) ) .
\end{equation*}
Moreover the last one converges to zero due to Lemma \ref{lem:est:tf}. We deduce that
\begin{align}
\int_{0}^{T}\int_{\Omega} \rho_{\eps} u_{\eps} \otimes u_{\eps} : \nabla \phi_{\eps} = & \, \int_{0}^{T}\int_{\Omega} \rho u \otimes u : (\nabla \psi \otimes \nabla \Delta^{-1}(\psi \rho) ) \label{2:2} \\
&\, +  \int_{0}^{T}\int_{\Omega}\fn_{\eps} \psi \rho_{\eps} u_{\eps} \otimes u_{\eps} : \nabla(\nabla \Delta^{-1}(\psi \rho_{\eps})) + c(\eps), \nonumber 
\end{align}
with $ c(\eps) \lra 0 $ as $ \eps \lra 0 $.
The next term is 
\begin{align*}
\int_{0}^{T}\int_{\Omega} \rho_{\eps}^{\gamma} \DIV(\phi_{\eps}) = & \, \int_{0}^{T}\int_{\Omega} \rho_{\eps}^{\gamma}  \nabla \psi \cdot \Phi_{\eps}[\nabla \Delta^{-1}(\psi \rho_{\eps})] \\
&\, +  \int_{0}^{T}\int_{\Omega}\fn_{\eps} \psi^2  \rho_{\eps}^{\gamma} \rho^{\eps}   \\
&\, +  \int_{0}^{T}\int_{\Omega}  \psi \rho_{\eps}^{\gamma}\left(     \DIV(  \Phi_{\eps}[\nabla \Delta^{-1}(\psi \rho_{\eps}))] - \fn_{\eps} \DIV( \nabla \Delta^{-1}(\psi \rho_{\eps}) ) \right).
\end{align*}
Following the idea of the proof of \eqref{strategy:1}, we deduce that the first term converges to 
\begin{equation*}
\int_{0}^{T}\int_{\Omega} \overline{\rho^{\gamma}}  \nabla \psi \cdot \nabla \Delta^{-1}(\psi \rho),
\end{equation*}
and the last term converges to zero from Lemma \eqref{lem:est:tf}. We deduce that 
\begin{align}
\int_{0}^{T}\int_{\Omega} \rho_{\eps}^{\gamma} \DIV(\phi_{\eps}) = & \, \int_{0}^{T}\int_{\Omega} \overline{\rho^{\gamma}}  \nabla \psi \cdot \nabla \Delta^{-1}(\psi \rho) \label{3:3} \\
&\, +  \int_{0}^{T}\int_{\Omega}\fn_{\eps} \psi^2   \rho_{\eps}^{\gamma} \rho^{\eps} + c(\eps), \nonumber
\end{align}
with $ c(\eps) \lra 0 $ as $ \eps \lra 0 $.
Finally 
\begin{align*}
\int_{0}^{T}\int_{\Omega} \bbS u_{\eps} : \nabla \phi_{\eps} = & \, \int_{0}^{T}\int_{\Omega} \bbS u_{\eps} : (\nabla \psi \otimes \Phi_{\eps}[\nabla \Delta^{-1}(\psi \rho_{\eps})]) \\
&\, +  \int_{0}^{T}\int_{\Omega}\fn_{\eps} \bbS u_{\eps} : D (\nabla \Delta^{-1}(\psi \rho_{\eps}))  \\
&\, +  \int_{0}^{T}\int_{\Omega}  \psi \rho_{\eps} u_{\eps} \otimes u_{\eps} :( \nabla \Phi_{\eps} -\fn_{\eps} \nabla )[\nabla \Delta^{-1}(\psi \rho_{\eps})].
\end{align*}
Using the strategy for \eqref{strategy:1} and Lemma \ref{lem:est:tf}, we have
\begin{align}
\int_{0}^{T}\int_{\Omega} \bbS u_{\eps} : \nabla \phi_{\eps} = & \, \int_{0}^{T}\int_{\Omega} \bbS u : (\nabla \psi \otimes \nabla \Delta^{-1}(\psi \rho) ) \label{4:4} \\
&\, +  \int_{0}^{T}\int_{\Omega}\fn_{\eps} \psi \bbS u_{\eps} : \nabla(\nabla \Delta^{-1}(\psi \rho_{\eps})) + c(\eps), \nonumber 
\end{align}
with $ c(\eps) \lra 0 $ as $ \eps \lra 0 $.
Using \eqref{1:1}-\eqref{2:2}-\eqref{3:3}-\eqref{4:4}, we rewrite \eqref{5:5}
\begin{align}
\lim_{\eps \to 0 }   \int_{0}^{T}\int_{\Omega} & \Bigg( -\psi  \fn_{\eps} \rho_{\eps} u_{\eps} \cdot \nabla \Delta^{-1}(\DIV(\psi \rho_{\eps} u_{\eps})) + \fn_{\eps} \psi \rho_{\eps} u_{\eps} \otimes u_{\eps} : \nabla(\nabla \Delta^{-1}(\psi \rho_{\eps})) \nonumber  \\
& \, \, \, \,  +\fn_{\eps} \psi^2   \rho_{\eps}^{\gamma} \rho^{\eps} - \fn_{\eps} \psi \bbS u_{\eps} : D(\nabla \Delta^{-1}(\psi \rho_{\eps})) \Bigg) \label{6:6}\\
= \int_{0}^{T}\int_{\Omega}&  \Bigg(- \psi \rho u \cdot \nabla\Delta^{-1}[\DIV(\psi \rho u )] +\psi \rho u \otimes u : \nabla^2 \Delta^{-1}[\psi \rho] \nonumber \\
& \, \, \, \,  + \psi^{2} \overline{\rho^{\gamma}} \rho - \psi \mathbb{S}u:  D \nabla \Delta^{-1}[\psi \rho]\Bigg). \nonumber 
\end{align}
We will now show that 
\begin{align}
\label{7:7}
\lim_{\eps \to 0 }   \int_{0}^{T}\int_{\Omega} &  -\psi  \fn_{\eps} \rho_{\eps} u_{\eps} \cdot \nabla \Delta^{-1}(\DIV(\psi \rho_{\eps} u_{\eps})) + \fn_{\eps} \psi \rho_{\eps} u_{\eps} \otimes u_{\eps} : \nabla(\nabla \Delta^{-1}(\psi \rho_{\eps})) \\
= & \, \int_{0}^{T}\int_{\Omega}- \psi \rho u \cdot \nabla\Delta^{-1}[\DIV(\psi \rho u )] +\psi \rho u \otimes u : \nabla^2 \Delta^{-1}[\psi \rho]. \nonumber
\end{align}
Using the strategy used in \eqref{strategy:1}, it is enough to show that
\begin{align*}
\lim_{\eps \to 0 }   \int_{0}^{T}\int_{\Omega} &  -\psi  \Phi_{\eps}[ \rho_{\eps} u_{\eps}] \cdot \nabla \Delta^{-1}(\DIV(\psi \rho_{\eps} u_{\eps})) +  \psi \Phi_{\eps}[\rho_{\eps} u_{\eps}] \otimes u_{\eps} : \nabla(\nabla \Delta^{-1}(\psi \rho_{\eps})) \\
= & \, \int_{0}^{T}\int_{\Omega}- \psi \rho u \cdot \nabla\Delta^{-1}[\DIV(\psi \rho u )] +\psi \rho u \otimes u : \nabla^2 \Delta^{-1}[\psi \rho].
\end{align*}
In the case $ \gamma > 2 $, this equality can be verified by using the commutator estimates from Step 3 of proof of Theorem 5.1 of \cite{Lions2}.

Finally notice that 
\begin{align*}
\int_{0}^{T} \int_{\Omega}\fn_{\eps} \psi \bbS u_{\eps}: \nabla^{2} \Delta^{-1}(\psi \rho^{\eps}) =  & \, \int_{0}^{T} \int_{\Omega}\fn_{\eps} \psi \mu D u_{\eps}: D \nabla \Delta^{-1}(\psi \rho^{\eps}) \\ 
& \, + \int_{0}^{T} \int_{\Omega}\fn_{\eps} \psi^2 (\mu +\lambda) \DIV( u_{\eps}) \rho^{\eps}.
\end{align*}
From some integrations by parts and using the density of smooth functions in Sobolev spaces, we have 
\begin{align*}
\int_{0}^{T} & \int_{\Omega}\fn_{\eps} \psi \mu D u_{\eps}: D \nabla \Delta^{-1}(\psi \rho^{\eps}) -  \int_{0}^{T} \int_{\Omega}\fn_{\eps} \psi^2  \mu \DIV( u_{\eps}) \rho^{\eps} \\ & \, =  \int_{0}^{T} \int_{\Omega} \psi \mu D u: D \nabla \Delta^{-1}(\psi \rho) -  \int_{0}^{T} \int_{\Omega} \psi^2  \mu \DIV( u) \rho + c(\eps)
\end{align*}
with $ c(\eps) \lra 0 $ as $ \eps \lra 0 $.
We deduce that
\begin{align}
\label{8:8}
\int_{0}^{T} & \int_{\Omega}\fn_{\eps} \psi \bbS u_{\eps}: \nabla^{2} \Delta^{-1}(\psi \rho^{\eps}) = \int_{0}^{T} \int_{\Omega}\fn_{\eps} \psi^2 (2\mu +\lambda) \DIV( u_{\eps}) \rho^{\eps} \\
& - \int_{0}^{T} \int_{\Omega} \psi \mu D u: D \nabla \Delta^{-1}(\psi \rho) + \int_{0}^{T} \int_{\Omega} \psi^2  \mu \DIV( u) \rho + c(\eps) .\nonumber 
\end{align}

The statement of the Lemma follows from \eqref{6:6}-\eqref{7:7} and \eqref{8:8}.

\end{proof}

\appendix

\section{The Bogovski\u{\i} operator in domains with holes}

In this appendix we recall a definition of Bogovski\u{\i} operator for domains with a hole. Moreover we show estimates independent of the size of the hole when it is assume to be small enough.

Let recall that a Bogovski\u{\i} operator is a left inverse of the divergence on $ \tilde{L}^{p} $ which is the space of $ L^{p} $ functions with integral zero. Due to the non-uniqueness of this operator, we choose $ \B_{\Omega} $ to satisfy the following extra property.

\begin{Theorem}

There exists a Bogovski\u{\i} operator $ \B_{\Omega} $ such that 
\begin{equation*}
 \B_{\Omega}: \tilde{L^{p}} \lra W^{1,p}_{0}(\Omega) 
\end{equation*}
and it is linear and continuous for any $  1 < p < +\infty$,
\begin{equation*}
\DIV(B_{\Omega}[f]) = f \text{ for any } f \in \tilde{L}^p(\Omega) \quad \text{ and } \quad  \|B_{\Omega}[f]\|_{L^{\infty}(\Omega)} \leq \|f\|_{L^2(\Omega)}.
\end{equation*}
Moreover for any vector field $ F \in L^{p}(\Omega) $ such that $ F\cdot n = 0 $ on $ \partial \Omega $, it holds 
\begin{equation*}
\|\B[\DIV(F)]\|_{L^{p}(\Omega)} \leq \|F\|_{L^p(\Omega)}.
\end{equation*}

\end{Theorem}
We refer to subsection 3.3.1.2 of \cite{NS} for a proof of the above theorem and for more details.

To define the Bogovski\u{\i} operator on the domain with hole $ \Omega \setminus B_{\eps}(0) $ we use an idea from \cite{lu:sch}, more precisely we define $ \B_{\Omega \setminus B_{\eps}(0)} $ as the composition of three operators. The extension by zero operator $ \calE_{\eps}: \tilde{L^{p}}(\Omega \setminus B_{\eps}(0))  \lra \tilde{L}^{p}(\Omega) $, the Bogovski\u{\i} operator on $ \Omega $ and the restriction operator $ \calR_{\eps} : W^{1,p}_{0}(\Omega) \lra W^{1,p}_{0}(\Omega \setminus B_{\eps}(0)) $ which is defined as follows.

Let $ \eta: [0,+\infty) \lra [0,1] $ an increasing smooth function such that $ \eta(x) = 0 $ for $ x \in [0,1] $ and $ \eta(x) = 1$ for $ x \in [2,+\infty) $ and let $ B_1 = \B_{B_{2}(0)\setminus B_{1}(0)} $ a  Bogovski\u{\i} operator.  For $ \eps > 0 $ let introduce the functions $ \eta_{\eps}(x) = \eta(x/\eps) $ and similarly $ B_{\eps}[f](x) = \eps B_1[f(\eps y)](x/\eps) $. We define the restriction operator 
\begin{equation*}
\calR_{\eps}[F] = \eta_{\eps} F +  B_{\eps}[\DIV(1-\eta)F)- \ll \DIV((1-\eta)F) \gg] ,
\end{equation*}
where
\begin{equation*}
 \ll  f \gg = \frac{1}{|\B_{2 \eps}(0) \setminus B_{\eps}(0)|} \int_{\B_{2 \eps}(0) \setminus B_{\eps}(0)} f.
\end{equation*}
We can define the the Bogovski\u{\i} operator on the domain with hole $ \Omega \setminus B_{\eps}(0) $.
\begin{equation}
\label{Bog:op:eps}
\B_{\Omega \setminus B_{\eps}(0)}[f] = \B_{\eps}[f] = \calR_{\eps} \circ \B_{\Omega} \circ \calE_{\eps}[f] .
\end{equation}
Moreover they satisfy the following estimates uniformly in $ \eps $.  

\begin{Proposition}

The operators $ B_{\eps} $ defined in \eqref{Bog:op:eps} are Bogovski\u{\i} operators, moreover for $ 1 < p \leq 2 $ they satisfy the uniform bounds
\begin{equation*}
\| B_{\eps}[f] \|_{W^{1,p}_{0}(\Omega \setminus B_{\eps}(0))} \leq C \|f\|_{L^{p}(\Omega \setminus B_{\eps}(0))} \quad \text{ and } \quad  \|B_{\eps}[f]\|_{L^{\infty}(\Omega \setminus B_{\eps}(0))} \leq C \|f\|_{L^2(\Omega \setminus B_{\eps}(0))},
\end{equation*} 
 with $ C $ independent of $ \eps $.  For any vector field $ F \in L^{q}(\Omega \setminus B_{\eps}(0)) $ such that $ F\cdot n = 0 $ on $ \partial \Omega \cup \partial B_{\eps}(0) $, it holds 
\begin{equation}
\label{div:bog:est}
\|\B[\DIV(F)]\|_{L^{q}(\Omega \setminus B_{\eps}(0))} \leq \|F\|_{L^q(\Omega \setminus B_{\eps}(0))}.
\end{equation}
for any $ 1 < q < + \infty $.

\end{Proposition}

\begin{proof}

The proof follows from the definition of the operator $ \B_{\eps}$. Compared with the correspondent result in \cite{lu:sch} we notice that \eqref{div:bog:est} holds also for  $ 2 \leq q < + \infty $. So let show this result. By definition, we have
\begin{align*}
\B_{\eps}[\DIV(F)] = & \, \calR_{\eps} \circ \B_{\Omega}[\DIV(F)] = \eta_{\eps}\B_{\Omega}[\DIV(F)] + B_{\eps}[\DIV((1-\eta_{\eps}) \B_{\Omega}[\DIV(F)])] \\
= & \, \eta_{\eps}\B_{\Omega}[\DIV(F)] - B_{\eps}[\nabla \eta_{\eps} \cdot \B_{\Omega}[\DIV(F)] ] +  B_{\eps}[\DIV((1-\eta_{\eps})F)] +  B_{\eps}[\nabla \eta_{\eps} \cdot F ]
\end{align*}
We estimate the right hand side separately. It is straight-forward to see that
\begin{equation*} 
\|\eta_{\eps}\B_{\Omega}[\DIV(F)]\|_{L^{q}(\calF_{\eps})} \leq C \|F\|_{L^{q}(\F_{\eps})}
\end{equation*}
For the second term we denote by $ q^{*} = 2q/(2-q) $ and we notice that the support of $ B_{\eps} $ is contained in $ B_{2\eps}(0) \setminus B_{\eps}(0) = A_{\eps}$. Then
\begin{align*}
\| B_{\eps}&[\nabla \eta_{\eps} \cdot \B_{\Omega}[\DIV(F)] ]\|_{L^{q}(A_{\eps})} \leq C\eps \| B_{\eps}[\nabla \eta_{\eps} \cdot \B_{\Omega}[\DIV(F)] ]\|_{L^{q^*}(A_{\eps})} \\ 
\leq & \, C\eps \| \nabla B_{\eps}[\nabla \eta_{\eps} \cdot \B_{\Omega}[\DIV(F)] ]\|_{L^{q}(A_{\eps})} \leq C \eps \|[\nabla \eta_{\eps} \cdot \B_{\Omega}[\DIV(F)]\|_{L^{q}(A_{\eps})} \\
 \leq & \, C \eps \|\nabla \eta_{\eps}\|_{L^{\infty}(A_{\eps})} \|\B_{\Omega}[\DIV(F)]\|_{L^{q}(A_{\eps})} \leq \|F\|_{L^{q}(\calF_{\eps})}.
\end{align*}
The same strategy gives
\begin{equation*}
\| B_{\eps}[\nabla \eta_{\eps} \cdot F ]\|_{L^{q}(A_{\eps})} \leq C \|F\|_{L^{q}(\F_{\eps})}.
\end{equation*}
We are left to show the estimates for $  B_{\eps}[\DIV((1-\eta_{\eps})F)] $. To simplify the notation let $ G = (1-\eta_{\eps})F $. By definition of $ B_{\eps} $ we have
\begin{align*}
B_{\eps}[\DIV_x(G)](x) = \eps B_1[\DIV_x(G(\eps y))](x/\eps) = B_{1}[\DIV_y(G(\eps y))](x/\eps) 
\end{align*}
We deduce that 
\begin{equation*}
\|B_{\eps}[\DIV_x(G)](x)\|_{L^{q}(A_{\eps})}= \eps^{2/q}\| B_{1}[\DIV_y(G(\eps y))]\|_{L^{q}(A_1)} \leq C  \eps^{2/q}\| G \|_{L^{q}(A_1)} = \|G\|_{L^{q}(A_{\eps})}.
\end{equation*}
After recalling that $ G = (1-\eta_{\eps})F  $ we obtain the desired result.

\end{proof}

\section{Improved pressure estimates}
\label{press:est:app}

This section is devoted to the proof of the improved pressure estimates from Proposition \ref{press:est:lem}.

\begin{proof}

The idea is to test the momentum equation of \eqref{CNS:equ} with 
\begin{equation}
\label{test:funct:impr}
\varphi_{\eps} = \phi \B_{\eps}[\psi_{\eps} \rho_{\eps}^{\theta} - \langle \psi_{\eps} \rho_{\eps}^{\theta} \rangle] ,
\end{equation}
where $ \varphi \in C^{\infty}_{c}([0,T)) $, $ \psi_{\eps} = (\tilde{\psi}_{\eps})^2$ with $ \tilde{\psi}_{\eps}(x) = 1 - \eta_{\eps}{|x|} $ and $$  \langle \psi_{\eps} \rho_{\eps}^{\theta} \rangle = \frac{1}{|\Omega \setminus B_{\eps}(0)|} \int_{\Omega \setminus B_{\eps}(0)}  \psi_{\eps} \rho_{\eps}^{\theta} .$$ 
The functions $ \varphi_{\eps}$ are not smooth enough in the time variable to be test functions in the weak formulations, so to be rigorous we should smooth them out by using a convolution kernel as in Section 7.9.5 of \cite{NS}. We will not consider this regularization here because it will not influence the estimates we are going to do. 

If we use \eqref{test:funct:impr} in the momentum equation of  \eqref{CNS:equ}, we deduce
\begin{align*}
\int_{0}^T \int_{\F_{\eps}} \phi \psi_{\eps} \rho_{\eps}^{\gamma+\theta} = & \, \int_{0}^T \int_{\F_{\eps}} \phi \rho^{\gamma} \langle \psi_{\eps} \rho_{\eps}^{\theta} \rangle -  \int_{\F_{\eps}} q_{\eps}^{in} \cdot \varphi_{\eps}(0,.) + 2 \mu \int_{0}^T \int_{\F_{\eps}} D u_{\eps} : D \varphi_{\eps} \\ &+ (\lambda + \mu) \int_{0}^T \int_{\F_{\eps}} \DIV(u_{\eps}) \DIV(\varphi_{\eps}) - \int_{0}^T \int_{\F_{\eps}} \rho_{\eps}u_{\eps} \otimes u_{\eps}: D \varphi_{\eps} \\ & - \int_{0}^T \int_{\F_{\eps}} \rho_{\eps} u_{\eps} \partial_t \varphi_{\eps} \\
= & \, \sum_{i=1}^6 I_i. 
\end{align*}
We will now show that the right hand side of the above expression is bounded by a constant independent of $ \eps $ multiply by the norm of the initial data.To do that we estimate the $ I_i $ separately. 

\begin{align*}
\left| I_1 \right| = \left|  \int_{0}^T \int_{\F_{\eps}} \phi \rho^{\gamma} \langle \psi_{\eps} \rho_{\eps}^{\theta} \rangle \right|  \leq C \|\rho_{\eps}\|_{L^{\infty}(0,T;L^{\gamma}(\F_{\eps}))} \|\rho_{\eps}\|_{L^{\infty}(0,T;L^{\gamma}(\F_{\eps}))}^{\theta} \leq C
\end{align*}
where we use $ \theta \leq \gamma $. Using the definition of $ \varphi_{\eps} $, we have

\begin{align*}
\left| I_2 \right| \leq & \,  \left|  \int_{\F_{\eps}} q_{\eps}^{in} \cdot \phi(0) \B_{\eps} \left[\psi_{\eps} (\rho_{\eps}^{in})^{\theta} - \langle \psi_{\eps} (\rho_{\eps}^{in})^{\theta} \rangle\right]  \right|  \\
\leq & \, \left\| \frac{q^{\in}_{\eps}}{\sqrt{\rho_{\eps}^{in}}}\right\|_{L^{2}(\F_{\eps})} \|\sqrt{\rho_{\eps}}\|_{L^{2\gamma}(\F_{\eps})}\|\B_{\eps}\left[(\psi_{\eps} \rho_{\eps}^{in})^{\theta} - \langle \psi_{\eps} (\rho_{\eps}^{in})^{\theta} \rangle\right]\|_{L^{2\gamma/(\gamma-1)}(\F_{\eps})} \\
\leq & C  \left\| \frac{q^{\in}_{\eps}}{\sqrt{\rho_{\eps}^{in}}}\right\|_{L^{2}(\F_{\eps})} \|\sqrt{\rho_{\eps}}\|_{L^{2\gamma}(\F_{\eps})}
\|(\psi_{\eps} \rho_{\eps}^{in})^{\theta} - \langle \psi_{\eps} (\rho_{\eps}^{in})^{\theta} \|_{L^{2\gamma/(2\gamma-1)}(\F_{\eps})} \\
\leq & \, C,
\end{align*}
for $ 2\theta \leq 2 \gamma -1$. In the third inequality we use that $$ \| \B_{\eps}[f]\|_{L^{p^*} } \leq \|\B_{\eps}[f]\|_{W^{1,p}} \leq \|f\|_{L^p } \quad \text{ for }  2 <  p^{*} = 2p/(2-p)  .$$
\begin{align*}
\left|I_3\right| \leq & \, \left| \mu \int_{0}^T \int_{\F_{\eps}}  D u_{\eps} : D \B_{\eps} \left[\phi \psi_{\eps} \rho_{\eps}^{\theta} - \langle \phi \psi_{\eps} \rho_{\eps}^{\theta} \rangle\right] \right| \\
\leq & \, C \| Du_{\eps} \|_{L^2(0,T;L^{2}(\F_{\eps}))}\|(\phi \psi_{\eps})^{1/\theta} \rho\|_{L^ {2\theta}(0,T;L^{2\theta}(\F_{\eps}))}^{1/2}.
\end{align*}
Notice that the last term of the right hand side can be absorb in the left hand side for any $ \theta \leq \gamma $.
\begin{align*}
\left|I_4 \right| \leq & \, \left| (\mu+\lambda) \int_{0}^T \int_{\F_{\eps}}  \DIV(u_{\eps}) \left[\phi \psi_{\eps} \rho_{\eps}^{\theta} - \langle \phi \psi_{\eps} \rho_{\eps}^{\theta} \rangle\right] \right| \\
\leq & \, C \|\DIV(u_{\eps}) \|_{L^2(0,T;L^{2}(\F_{\eps}))}\|(\phi \psi_{\eps})^{1/\theta} \rho\|_{L^ {2\theta}(0,T;L^{2\theta}(\F_{\eps}))}^{1/2}.
\end{align*}
As before we can absorb the last term of the right hand side in the left hand side if $ \theta \leq \gamma $. The next term is
\begin{align*}
\left|I_5 \right| \leq & \, \left| \int_{0}^T \int_{\F_{\eps}}  \rho_{\eps} u_{\eps} \otimes u_{\eps} : D \B_{\eps} \left[\phi \psi_{\eps} \rho_{\eps}^{\theta} - \langle \phi \psi_{\eps} \rho_{\eps}^{\theta} \rangle\right] \right| \\
\leq & \, \|\rho_{\eps}\|_{L^{\infty}(0,T; L^{\gamma}(\F_{\eps}))}\|u\|_{L^2(0,T;L^{\infty^-}(\calF_{\eps}))}^2\|\rho^{\theta}\|_{L^{\infty}(0,T;L^{\gamma/\theta}(\calF_{\eps}))}.
\end{align*}
in the above inequality we need $ \gamma/\theta \leq  2 $ for the estimates on $ D \B_{\eps} $, moreover from 
$$ \frac{1}{\gamma} + \frac{1}{\infty^-} +  \frac{1}{\infty^-} + \frac{\theta}{\gamma} \leq 1 $$  
we have the classical bound $ \theta < \gamma -1 $ and from $ \gamma/\theta \leq 2 $ we have also $ \gamma > 2 $.

We are now left with the estimates of $ I_6 $. Notice that for any fixed $ \eps $  we already know that $ \rho_{\eps} \in L^{q}((0,T)\times \F_{\eps}) $ for any $ q < 2\gamma -1 $. Lemma 6.9 of \cite{NS} ensures that $ \rho_{\eps} $ and $ u_{\eps} $ satisfy the equation
\begin{equation*}
\partial_t \rho_{\eps}^{\theta} + \DIV(u_{\eps} \rho^{\theta}_{\eps}) + \DIV(u_{\eps})(\theta-1)\rho^{\theta}= 0,
\end{equation*}
in a distributional sense for any $ \theta < \gamma -1/2 $. Using the equation we have that 
\begin{align*}
\partial_t(\varphi_{\eps}) = & \, \partial_t \phi \B_{\eps}\left[ \psi_{\eps} \rho_{\eps}^{\theta} - \langle  \psi_{\eps} \rho_{\eps}^{\theta} \rangle\right] +  \phi \B_{\eps}\left[ \psi_{\eps} \partial_t \rho_{\eps}^{\theta} - \langle   \psi_{\eps} \partial_t \rho_{\eps}^{\theta} \rangle\right] \\ 
= & \,  \partial_t \phi \B_{\eps}\left[ \psi_{\eps} \rho_{\eps}^{\theta} - \langle  \psi_{\eps} \rho_{\eps}^{\theta} \rangle\right]- \phi \B_{\eps}\left[ \DIV( \psi_{\eps} \cdot u_{\eps} \rho_{\eps}^{\theta} ) - \langle  \DIV( \psi_{\eps} \cdot u_{\eps} \rho_{\eps}^{\theta}) \rangle\right]  \\ & \, + \phi \B_{\eps}\left[ \nabla \psi_{\eps} \cdot u_{\eps}  \rho_{\eps}^{\theta} - \langle  \nabla \psi_{\eps} \cdot \cdot u_{\eps}  \rho_{\eps}^{\theta} \rangle\right] \\ & - \phi \B_{\eps}\left[  \psi_{\eps}\DIV( u_{\eps})(\theta-1) \rho_{\eps}^{\theta}  - \langle  \psi_{\eps} \DIV(  u_{\eps})(\theta-1) \rho_{\eps}^{\theta} \rangle\right] \\
= & \, \sum_{j = 1}^{4} J_{j}.
\end{align*}
We can now estimate 
\begin{equation*}
\left|I_{6}\right| \leq \sum_{j=1}^{4} \left|\int_{0}^{T}\int_{\F_{\eps}} \rho_{\eps}u_{\eps} \cdot J_{j}  \right| \leq \|\rho_{\eps}\|_{L^{\infty}(0,T;L^{\gamma}(\F_{\eps}))}\|u_{\eps}\|_{L^{2}(0,T; L^{\infty^-}(\F_{\eps}))} \|J_j\|_{L^{2}(0,T; L^{(\gamma/\theta)-}(\F_{\eps}))}. 
\end{equation*}
We are left with the estimates of $  J_j $. Notice that
\begin{equation*}
 \|J_1\|_{L^{2}(0,T; L^{(\gamma/\theta)-}(\F_{\eps}))} \leq C \| \partial_t \phi \|_{L^{2}(0,T)} \|\rho_{\eps}^{\theta}\|_{L^{\infty}(0,T;L^{\gamma/\theta}(\F_{\eps}))}
\end{equation*}
Then we have
\begin{align*}
 \|J_2\|_{L^{2}(0,T; L^{(\gamma/\theta)-}(\F_{\eps}))} \leq & \,   C \| u_{\eps} \rho_{\eps}^{\theta}\|_{L^{\infty}(0,T;L^{(\gamma/\theta)-}(\F_{\eps}))} 
 \leq & \, \| u_{\eps}^{\theta}\|_{L^{2}(0,T;L^{\infty^-}(\F_{\eps}))} \|\rho_{\eps}^{\theta}\|_{L^{\infty}(0,T;L^{\gamma/\theta}(\F_{\eps}))}
\end{align*}

Using that 
$$ \left(\left( \frac{2\gamma}{\gamma+ 2\theta }\right)^- \right)^* = \left(\frac{\gamma}{\theta}\right)^{-} \quad \text{ for } \gamma > 2\theta ,$$
we have
\begin{align*}
\|J_3\|_{L^{2}(0,T; L^{(\gamma/\theta)-}(\F_{\eps}))} \leq & C \| \DIV( u_{\eps}) \psi_{\eps}^{\theta} \|_{L^{2}(0,T;L^{2\gamma/(\gamma + 2\theta)}(\F_{\eps})} \\
\leq & \, C \|u_{\eps}\|_{L^{2}(0,T;L^{2}(\F_{\eps}))} \|\rho_{\eps}^{\theta}\|_{L^{\infty}(0,T;L^{\gamma/\theta}(\F_{\eps}))}
\end{align*}

and similarly
\begin{align*}
\|J_4\|_{L^{2}(0,T; L^{(\gamma/\theta)-}(\F_{\eps}))} \leq & C \| \nabla\psi_{\eps} u_{\eps} \psi_{\eps}^{\theta} \|_{L^{2}(0,T;L^{(2\gamma/(\gamma + 2\theta)^-}(\F_{\eps})} \\
\leq & \, C\| \nabla \psi_{\eps} \|_{L^{2}(\F_{\eps})} \|u_{\eps}\|_{L^{2}(0,T;L^{\infty^-}(\F_{\eps}))} \|\rho_{\eps}^{\theta}\|_{L^{\infty}(0,T;L^{\gamma/\theta}(\F_{\eps}))}
\end{align*}

For $ \theta = \gamma/2 $, we use a different estimate. 
First of all notice that from interpolation we have
\begin{equation*}
\|\rho_{\eps} u_{\eps}\|_{L^{6}(0,T; L^{(3/2)^+}(\F_{\eps}))} \leq \|\rho_ {\eps} u_{\eps}\|_{L^{2}(0,T;L^{\gamma^-}(\F_{\eps}))}^{1/3} \|\rho_{\eps} u_{\eps}\|_{L^{\infty}(0,T;L^{2\gamma/(\gamma+1)})}^{2/3}
\end{equation*}
under the hypothesis $ \gamma > 2 $. For $ j = 3, 4 $, we have
\begin{equation*}
\left| \int_0^{T} \int_{\F_{\eps}} \rho_{\eps} u_{\eps} \cdot J_{j} \right| \leq \|\rho_{\eps} u_{\eps} \|_{L^{6}(0,T; L^{(3/2)^+}(\F_{\eps}))}\|J_j\|_{L^{6/5}(0,T; L^{3^-}(\F_{\eps})}
\end{equation*}
We then estimate
\begin{align*}
\|J_3\|_{L^{6/5}(0,T; L^{3^-}(\F_{\eps})} \leq \| \DIV(u_{\eps}) \|_{L^{2}(0,T;L^{2}(\calF_{\eps})} \|\phi \psi_{\eps} \rho^{\gamma/2}_{\eps}\|_{L^{3}(0,T;L^{3}(\F_{\eps}))}
\end{align*}
in particular we can absorb the last term on the right hand side.

Recall that we assume $ \psi_{\eps} = \tilde{\psi}_{\eps}^2 $. Similarly
\begin{align*}
\|J_4\|_{L^{6/5}(0,T; L^{3^-}(\F_{\eps})} \leq & C \eps \| \nabla\psi_{\eps} u_{\eps} \psi_{\eps}^{\gamma/2} \|_{L^{2}(0,T;L^{2^-}(\F_{\eps})} \\
\leq & \, C\| \nabla \tilde{\psi}_{\eps} \|_{L^{2}(\F_{\eps})} \|u_{\eps}\|_{L^{2}(0,T;L^{\infty^-}(\F_{\eps}))} \|\phi \tilde{\psi}_{\eps} \rho_{\eps}^{\gamma/2}\|_{L^{3}(0,T;L^{3}(\F_{\eps}))}.
\end{align*}
 
\end{proof}

\section{Proof of Proposition \ref{prop:cut:off}}
\label{app:Jiao}

In this section we prove Proposition \ref{prop:cut:off} which is a straight-forward extension of Lemma 3 of \cite{He:2D}. First of all for $ A, B \in \R $ with $ 0 < A < B $, we denote by $ \alpha = B/ A > 1 $ and we define the functions
\begin{equation*}
f_{A,B}(z) = \begin{cases}
1 \quad & \text{ for } 0\leq z < A, \\
\frac{\log z - \log B }{\log A - \log B } \quad & \text{ for } A \leq z \leq B, \\
0 \quad & \text{ for } z > B.
\end{cases}
\end{equation*} 
It holds that $ f_{A,B} \in W^{1,\infty}(\R^{+}) $. We define the cut-off 
\begin{equation*}
\tilde{\eta}_{\eps, \alpha_{\eps}}(x) = f_{\eps, \alpha_{\eps} \eps }(|x|),  
\end{equation*} 
for $ x \in \bbR^2 $ and  $ \alpha_{\eps}  > 1 $.

\begin{Proposition}
\label{prop:cut:off}
Under the hypothesis that $  \alpha_{\eps} \leq |\log(\eps)| $ and $ \alpha_{\eps} \lra +\infty $, it holds

\begin{enumerate}

\item The functions $  1- \tilde{\eta}_{\eps, \alpha_{\eps}}   \longrightarrow 0 $  in $ L^{q}(\R^2) $ for $ 1 \leq q < +\infty $.

\item For $ 1\leq q < 2 $, $$ \left\| \nabla \tilde{\eta}_{\eps, \alpha_{\eps}} \right\|_{L^{q}(\R^2)}^q = \frac{2\pi}{2-q}  \frac{\alpha_{\eps}^{2-q}-1}{(\log \alpha_{\eps})^{q} } \eps^{2-q}. $$

\item We have $$ \left\| \nabla \tilde{\eta}_{\eps, \alpha_{\eps}} \right\|_{L^{2}(\R^2)}^2 = \frac{2\pi}{\log \alpha_{\eps} }. $$

\item For  $  2 <  q < +\infty $, for $ i = 1, 2 $,  $$ \left\| \nabla \tilde{\eta}_{\eps, \alpha_{\eps}} \right\|_{L^{q}(\R^2)}^q + \left\| \nabla^2 \tilde{\eta}_{\eps, \alpha_{\eps}} x_i \right\|_{L^{q}(B_{\eps \alpha_{\eps}}(0))}^q = \frac{2\pi}{q-2}  \frac{1}{(\log \alpha_{\eps})^{q} }\frac{\eps^2}{ \eps^{q}} \left(1-\frac{1}{(\alpha_{\eps})^{q-2}} \right). $$

\end{enumerate}
In particular 
 $$ \left\| \nabla \tilde{\eta}_{\eps, \alpha_{\eps}} \right\|_{L^{q}(\R^2)} \lra 0 \quad \text{ for } 1 \leq q \leq 2 $$
and 
$$  \eps \alpha_{\eps} \left\| \nabla \tilde{\eta}_{\eps, \alpha_{\eps}} \right\|_{L^{q}(\R^2)}, \eps \alpha_{\eps} \left\| \nabla^2 \tilde{\eta}_{\eps, \alpha_{\eps}} x_i \right\|_{L^{q}(B_{\eps \alpha_{\eps}}(0))} \lra 0 \quad \text{ for }  2 <  q < +\infty.   $$

\end{Proposition}

\begin{proof}
After passing to radial coordinates the proof is straight-forward. For example to show part $ 3. $, we compute
\begin{align*}
 \left\| \nabla \tilde{\eta}_{\eps, \alpha_{\eps}} \right\|_{L^2(\bbR^2)}^2 = & \, \int_0^{2\pi} \int_{\eps}^{\eps \alpha_{\eps}} \left| \frac{1}{r} \frac{1}{\log (\eps) - \log(\eps \alpha_{\eps})}\right|^2  r  \, dr d\theta  \\
= & \, \frac{2 \pi}{(\log(\alpha_{\eps}))^2}\left[\log(r)\right]_{\eps}^{\eps \alpha_{\eps}} \\
= & \, \frac{4 \pi}{\log(\alpha_{\eps})}.
\end{align*} 
\end{proof}

The cut-off $  \tilde{\eta}_{\eps, \alpha_{\eps}} $ satisfy all the bounds of Proposition \ref{prop:cut:off} but they are not smooth, in particular they are not $ C^2 $ on $ \partial B_{\eps}(0) \cup \partial B_{\eps \alpha_{\eps}}(0) $. To solve this issue we modify these functions as in $  \cite{He:2D} $.
 Let introduce a function $ g \in C^{\infty}_{c}([0,12/10)) $ such that $ 0 \leq  g \leq  1 $ and $ g (y) = 1  $ for $ y \in [0,11/10] $. Then we define 

\begin{equation}
\label{family:of:cut-off}
\eta_{\eps, \alpha_{\eps}}(x) = 1 +\left( 1 - g\left( \frac{|x|}{\eps}\right) \right) \left( \tilde{\eta}_{\eps, \alpha_{\eps}}(x) g \left(  \frac{13}{10} \frac{|x|}{\alpha_{\eps}{\eps}}\right) -1\right),
\end{equation}  
 which rewrites 
 \begin{equation*}
 \eta_{\eps, \alpha_{\eps}}(x) = \begin{cases} 1 & \quad \text{ for } \quad   |x| < \frac{11}{10}\eps, \\   
 1 +\left( 1 - g\left( \frac{|x|}{\eps}\right) \right) \left( \tilde{\eta}_{\eps, \alpha_{\eps}}(x)-1\right) & \quad \text{ for}  \quad   \frac{11}{10}\eps \leq |x| < \frac{12}{10}\eps, \\ 
 \tilde{\eta}_{\eps, \alpha_{\eps}}(x)  & \quad \text{ for }  \quad   \frac{12}{10}\eps \leq |x| < \frac{11}{13}\eps \alpha_{\eps}, \\
  \tilde{\eta}_{\eps, \alpha_{\eps}}(x) g \left(  \frac{13}{10} \frac{|x|}{\alpha_{\eps}{\eps}}\right) & \quad \text{ for }  \quad   \frac{11}{13}\eps \alpha_{\eps} \leq |x| < \frac{12}{13}\eps \alpha_{\eps}, \\
0 & \quad \text{ for }  \quad   | x | \geq   \frac{12}{13}\eps \alpha_{\eps}.
 \end{cases}
 \end{equation*}
The functions $  \eta_{\eps, \alpha_{\eps}} $ are smooth. It remains to show that they satisfy all the properties stated in Proposition \ref{prop:cut:off}.

\begin{proof}[Proof of Proposition \ref{prop:cut:off}] We verify that the family $ \eta_{\eps, \alpha_{\eps}} $ defined in \eqref{family:of:cut-off} satisfies all the properties stated in Proposition \ref{prop:cut:off}. First of all by definition $\eta_{\eps, \alpha_{\eps}}  \in C^{\infty}_c(B_{\eps \alpha_{\eps}}(0)) $, $ \eta_{\eps, \alpha_{\eps}}(x) = 1 $ for $ x \in B_{\eps}(0) $  and $ \| \eta_{\eps, \alpha_{\eps}} \|_{L^{\infty}} \leq 1 $. Let now bound the $ L^{q} $ norm of $ \nabla \eta_{\eps, \alpha_{\eps}} $. As in $  \cite{He:2D} $, we denote by
$$ g^1_{\eps}(x) = \left( 1 - g\left( \frac{|x|}{\eps}\right) \right), \quad g^2_{\eps}(x) =  g \left(  \frac{13}{10} \frac{|x|}{\alpha_{\eps}{\eps}}\right)$$
and by $ A_{r,R} = B_R(0) \setminus B_r(0) $ the annulus for $ 0 < r < R $. Finally we notice that
 $$ \| \tilde{\eta}_{\eps, \alpha_{\eps}} - 1 \|_{L^{\infty}\left(A_{\frac{11}{10}\eps, \frac{12}{10}\eps}\right)} = \left\|  \frac{\log (|x|/\eps)}{\log(\alpha_{\eps})}\right\|_{L^{\infty}\left(A_{\frac{11}{10}\eps, \frac{12}{10}\eps}\right)} \leq \frac{C}{\log(\alpha_{\eps})} $$
 and similarly
  $$ \| \tilde{\eta}_{\eps, \alpha_{\eps}} \|_{L^{\infty}\left(A_{\frac{11}{13}\eps \alpha_{\eps}, \frac{12}{13}\eps \alpha_{\eps} }\right)} = \left\|  \frac{\log (|x|/(\eps\alpha_{\eps}))}{\log(\alpha_{\eps})}\right\|_{L^{\infty}\left(A_{\frac{11}{13}\eps \alpha_{\eps}, \frac{12}{10}\eps \alpha_{\eps}}\right)} \leq \frac{C}{\log(\alpha_{\eps})}. $$
 
For $  1 \leq q < +\infty$, we estimate
\begin{align*}
 \| \nabla \eta_{\eps, \alpha_{\eps}} \|_{L^{q}(\R^2)} \leq & \, \| \nabla(g_{\eps}^1 (\tilde{\eta}_{\eps, \alpha_{\eps}}-1) )\|_{L^{q}\left(A_{\frac{11}{10}\eps, \frac{12}{10}\eps}\right)} + \| \nabla \tilde{\eta}_{\eps, \alpha_{\eps}} \|_{L^{q}\left(A_{\frac{12}{10}\eps, \frac{11}{13}\eps \alpha_{\eps}}\right)} \\
 &\,  +   \| \nabla(g_{\eps}^2 \tilde{\eta}_{\eps, \alpha_{\eps}}) \|_{L^{q}\left(A_{\frac{11}{13}\eps \alpha_{\eps}, \frac{12}{10}\eps \alpha_{\eps}}\right)} \\
 \leq & \, \| \nabla \tilde{\eta}_{\eps, \alpha_{\eps}}\|_{L^{q}(\R^2)} \left( \| g_{\eps}^1 \|_{L^{\infty}(\R^2)} + 1 + \| g_{\eps}^2 \|_{L^{\infty}(\R^2)}      \right)  \\
 & \,  \| \tilde{\eta}_{\eps, \alpha_{\eps}} - 1 \|_{L^{\infty}\left(A_{\frac{11}{10}\eps, \frac{12}{10}\eps}\right)} \|\nabla g_{\eps}^1\|_{L^{q}(\bbR^2)} + \| \tilde{\eta}_{\eps, \alpha_{\eps}} \|_{L^{\infty}\left(A_{\frac{11}{13}\eps \alpha_{\eps}, \frac{12}{13}\eps \alpha_{\eps} }\right)}\|\nabla g_{\eps}^2\|_{L^{q}(\bbR^2)} \\
 \leq & \, C \| \nabla \tilde{\eta}_{\eps, \alpha_{\eps}}\|_{L^{q}(\R^2)} + \frac{C}{\log(\alpha_{\eps})}( (\alpha_{\eps} \eps)^{(2-q)/q} + \eps^{(2-q)/q}),
 \end{align*}
 where we use that $ 1 - g_{\eps}^1 $ and $ g_{\eps}^2$  are appropriate rescaling of $ g $ to estimate the $ L^q$ norm of $ \nabla g_{\eps}^1 $ and $ \nabla g_{\eps}^2 $. The bounds of the $ L^q $ norm of $ \nabla \eta_{\eps, \alpha_{\eps}} $ follows from the above estimate and Proposition \ref{prop:cut:off}, after noticing that in the case $ q \leq 2 $ it holds $ (\alpha_{\eps} \eps)^{(2-q)/q} \geq \eps^{(2-q)/q} $,  while for $ q \geq 2 $  it holds $ (\alpha_{\eps} \eps)^{(2-q)/q} \leq  \eps^{(2-q)/q} $. This explain the slightly different bounds in point 3. and 4.  
 Similarly we estimate for $ q \neq 2 $
\begin{align*}
\|  |x|  \nabla \eta_{\eps, \alpha_{\eps}} \|_{L^{q}(\R^2)} \leq & \, \| |x| \nabla (g_{\eps}^1 (\tilde{\eta}_{\eps, \alpha_{\eps}}-1) )\|_{L^{q}\left(A_{\frac{11}{10}\eps, \frac{12}{10}\eps}\right)} + \| |x| \nabla \tilde{\eta}_{\eps, \alpha_{\eps}} \|_{L^{q}\left(A_{\frac{12}{10}\eps, \frac{11}{13}\eps \alpha_{\eps}}\right)} \\
 &\,  +   \| |x|\nabla (g_{\eps}^2 \tilde{\eta}_{\eps, \alpha_{\eps}}) \|_{L^{q}\left(A_{\frac{11}{13}\eps \alpha_{\eps}, \frac{12}{10}\eps \alpha_{\eps}}\right)} \\
\leq & \, \| |x| \nabla \tilde{\eta}_{\eps, \alpha_{\eps}}\|_{L^{q}(\R^2)} \left( \| g_{\eps}^1 \|_{L^{\infty}(\R^2)} + 1 + \| g_{\eps}^2 \|_{L^{\infty}(\R^2)}      \right)  \\
 & \,  + \| \tilde{\eta}_{\eps, \alpha_{\eps}} - 1 \|_{L^{\infty}\left(A_{\frac{11}{10}\eps, \frac{12}{10}\eps}\right)} \||x|\nabla g_{\eps}^1\|_{L^{q}(\bbR^2)} \\
 & + \| \tilde{\eta}_{\eps, \alpha_{\eps}} \|_{L^{\infty}\left(A_{\frac{11}{13}\eps \alpha_{\eps}, \frac{12}{13}\eps \alpha_{\eps} }\right)}\||x|\nabla g_{\eps}^2\|_{L^{q}(\bbR^2)} \\
 \leq & C\left(  \| |x| \nabla \tilde{\eta}_{\eps, \alpha_{\eps}}\|_{L^{q}(\R^2)} \right) +  \frac{C}{\log(\alpha_{\eps})}( (\alpha_{\eps} \eps)^{2/q} + \eps^{2/q}).
\end{align*}
 where we use that $ 1 - g_{\eps}^1 $ and $ g_{\eps}^2$  are appropriate rescaling of $ g $ to estimate the $ L^{\infty}$ norm of $ |x| \nabla g_{\eps}^1 $ and $ |x| \nabla g_{\eps}^2 $. Finally we estimate for $ q \neq 2 $
\begin{align*}
\|  |x|  \nabla^2 \eta_{\eps, \alpha_{\eps}} \|_{L^{q}(\R^2)} \leq & \, \| |x| \nabla^2(g_{\eps}^1 (\tilde{\eta}_{\eps, \alpha_{\eps}}-1) )\|_{L^{q}\left(A_{\frac{11}{10}\eps, \frac{12}{10}\eps}\right)} + \| |x| \nabla^2 \tilde{\eta}_{\eps, \alpha_{\eps}} \|_{L^{q}\left(A_{\frac{12}{10}\eps, \frac{11}{13}\eps \alpha_{\eps}}\right)} \\
 &\,  +   \| |x|\nabla^2 (g_{\eps}^2 \tilde{\eta}_{\eps, \alpha_{\eps}}) \|_{L^{q}\left(A_{\frac{11}{13}\eps \alpha_{\eps}, \frac{12}{10}\eps \alpha_{\eps}}\right)} \\
\leq & \, \| |x| \nabla^2 \tilde{\eta}_{\eps, \alpha_{\eps}}\|_{L^{q}(B_{\eps \alpha_{\eps}}(0))} \left( \| g_{\eps}^1 \|_{L^{\infty}(\R^2)} + 1 + \| g_{\eps}^2 \|_{L^{\infty}(\R^2)}      \right)  \\
 & \,+ \| \nabla \tilde{\eta}_{\eps, \alpha_{\eps}}\|_{L^{q}(\R^2)} \left( \| |x| \nabla g_{\eps}^1 \|_{L^{\infty}(\R^2)} + \| |x| \nabla g_{\eps}^2 \|_{L^{\infty}(\R^2)}      \right)  \\
 & \,  + \| \tilde{\eta}_{\eps, \alpha_{\eps}} - 1 \|_{L^{\infty}\left(A_{\frac{11}{10}\eps, \frac{12}{10}\eps}\right)} \||x|\nabla^2 g_{\eps}^1\|_{L^{q}(\bbR^2)} \\
 & + \| \tilde{\eta}_{\eps, \alpha_{\eps}} \|_{L^{\infty}\left(A_{\frac{11}{13}\eps \alpha_{\eps}, \frac{12}{13}\eps \alpha_{\eps} }\right)}\||x|\nabla^2 g_{\eps}^2\|_{L^{q}(\bbR^2)} \\
 \leq & C\left(  \| |x| \nabla^2 \tilde{\eta}_{\eps, \alpha_{\eps}}\|_{L^{q}(B_{\eps \alpha_{\eps}}(0))} + \| \nabla \tilde{\eta}_{\eps, \alpha_{\eps}}\|_{L^{q}(\R^2)} \right) \\
 & \, +  \frac{C}{\log(\alpha_{\eps})}( (\alpha_{\eps} \eps)^{(2-q)/q} + \eps^{(2-q)/q}).
\end{align*}
 where as before we use that $ 1 - g_{\eps}^1 $ and $ g_{\eps}^2$  are appropriate rescaling of $ g $ to estimate the $ L^q$ norm of $ |x| \nabla^2 g_{\eps}^1 $ and $ |x|\nabla^2 g_{\eps}^2 $.

\end{proof}

\section*{Acknowledgements}

{\small M.B. is supported by the NWO grant OCENW.M20.194.}

\end{document}